\documentclass[11pt, twoside,a4paper]{article}
\usepackage{flafter,amsmath,amssymb,latexsym,psfrag,graphicx,color,indentfirst,bm,amsthm}
\usepackage[T1]{fontenc}
\usepackage[utf8]{inputenc}
\usepackage[colorlinks,
linktocpage=true,
linkcolor=blue,
citecolor=green,
urlcolor=blue,
anchorcolor=black
]{hyperref}
\usepackage[a4paper,top=1.8cm,bottom=1.8cm,left=1.8cm,right=1.5cm, marginparwidth=1.75cm]{geometry}
\usepackage{float}
\usepackage[makeroom]{cancel}
\usepackage{mathtools}
\usepackage[numbers,sort&compress]{natbib}
\usepackage{appendix}
\usepackage{multirow}
\usepackage{caption}
\usepackage{subfigure}
\usepackage{graphicx}
\usepackage{dutchcal}
\usepackage{comment}
\allowdisplaybreaks
\usepackage{tikz}
\tikzset{mynode/.style={inner sep=2pt,fill,outer sep=0,circle}}
\usetikzlibrary{fadings}
\usepackage{amsmath,amssymb,braket,tikz}
\usetikzlibrary{tikzmark,calc}
\usepackage{ebgaramond}
\usepackage[sfdefault]{FiraSans}
\usepackage{microtype}

\newtheorem{theorem}{Theorem}[section]
\newtheorem{lemma}{Lemma}[section] 
\newtheorem{corollary}{Corollary}[section] 
\newtheorem{proposition}{Proposition}[section] 
\newtheorem{remark}{Remark}[section]  

\numberwithin{equation}{section}

\usepackage{mathrsfs}
\newsavebox\foobox
\newlength{\foodim}

\numberwithin{equation}{section}
\numberwithin{figure}{section}

\DeclareRobustCommand{\rchi}{{\mathpalette\irchi\relax}}
\newcommand{\irchi}[2]{\raisebox{\depth}{$#1\chi$}}
\newcommand\obullet[1]{\ThisStyle{\ensurestackMath{%
  \stackon[1pt]{\SavedStyle#1}{\SavedStyle\kern.6\LMpt\bullet}}}}
\newcommand\ocirc[1]{\ThisStyle{\ensurestackMath{%
  \stackon[1pt]{\SavedStyle#1}{\SavedStyle\kern.6\LMpt\circ}}}}
  
\title{ 

Dispersive Effective Metasurface Model for Bubbly Media}
\author{Arpan Mukherjee\footnote{Joint Research Center of Applied Mathematics, Shenzhen MSU-BIT University, Shenzhen, People's Republic of China (arpanmath99@alumni.iitm.ac.in).} \ and Mourad Sini\footnote{Radon Institute (RICAM), Austrian Academy of Sciences, Altenbergerstrasse 69, A-4040, Linz, Austria (mourad.sini@oeaw.ac.at). This author is partially supported by the Austrian Science Fund (FWF): P36942.}} 

\begin{document}

\maketitle
\begin{abstract}   

We derive the effective transmission condition for a cluster of acoustic subwavelength resonators, modeled as small-scaled bubbles distributed not necessarily periodically along a smooth, bounded hypersurface, which need not be flat. The transmission condition specifies that the jump in the normal derivative of the acoustic field is proportional to its second time derivative, convoluted in time with a sinusoidal kernel. This kernel has a period determined by the common subwavelength resonance (specifically, the Minnaert resonance in this case).
\vspace{0.1cm}

\noindent
This dispersive transmission condition can also be interpreted as a Dirac-like surface potential that is convoluted in the time domain and spatially supported on the specified hypersurface. We highlight the following features:

\begin{enumerate}
      \item \textit{High resonance regime}: When the common resonance is large, the surface behaves as fully transparent, permitting complete transmission of the acoustic field.

      \item \textit{Moderate resonance regime}: For moderate resonance values, the surface acts as a screen with memory effects, capturing the dispersive behavior induced by the resonance.

      \item \textit{Low resonance regime}: When the common resonance is small, the surface functions as a partial reflective (or partial transmissive) screen with no memory effect.
\end{enumerate}

\end{abstract}    


   \section{Introduction}    

      \subsection{Motivation}    

\noindent
The study of acoustic subwavelength resonators has attracted significant attention due to their remarkable ability to manipulate sound waves at scales much smaller than the wavelength. Among these, bubble-based resonators represent a particularly compelling example, as their subwavelength resonances—such as the Minnaert resonance—enable unique acoustic phenomena, \cite{Commander-Prosperetti, Leroy-eta-l}. In many applications, these resonators are distributed on smooth, bounded hypersurfaces, which may be curved or non-flat, adding complexity to the resulting wave interactions. Understanding how such distributions affect acoustic transmission is critical for applications in acoustic Metamaterials, sound insulation, and wave control. Most of published works assume the surface to be flat and the resonators are distributed periodically, \cite{Ammari-et-al-0, Ammari-et-al-2, Maurel-1, Maurel-2} for instance.
In this work, we characterize the effective transmission condition for a cluster of such resonators, modeled as small-scale bubbles distributed along a non-flat smooth and bounded surface. This condition describes how the acoustic field interacts with a hypersurface populated by the resonators. Specifically, the transmission condition relates the jump in the normal derivative of the acoustic field across the hypersurface to its second time derivative, convoluted in time with a sinusoidal kernel. The period of this kernel is governed by the common subwavelength resonance of the bubbles, with the Minnaert resonance serving as a key example.
Importantly, this dispersive transmission condition can also be interpreted through a physical analogy: it behaves as a Dirac-like surface potential that is convoluted in the time domain and supported spatially on the hypersurface. This dual perspective provides both mathematical clarity and physical intuition for understanding the influence of resonator clusters on sound propagation.
To explore the effects of the common resonance, we identify three distinct regimes of acoustic behavior:
\begin{enumerate}
      \item	\textit{High resonance regime}: When the common resonance is large, the hypersurface behaves as a fully transparent medium, allowing complete transmission of the acoustic field.
      \item \textit{Moderate resonance regime}: In this intermediate case, the hypersurface exhibits memory effects, reflecting the dispersive characteristics induced by the resonance. These memory effects introduce complex time-dependent interactions, which are of particular interest for applications requiring precise wave control.
      \item \textit{Low resonance regime}: When the common resonance is small, the hypersurface acts as a conventional reflective or transmissive screen, akin to traditional acoustic barriers.
\end{enumerate}
\noindent
The findings presented here highlight the rich interplay between resonance, dispersion, and surface geometry in shaping acoustic wave behavior. By developing a deeper understanding of these effects, this study lays the foundation for advancements in the design of novel acoustic devices, ranging from highly transparent surfaces to tunable acoustic screens with memory-dependent properties.


       \subsection{Problem setting and the mathematical model}  

\noindent       
Consider \( D \) to be a \( \mathcal{C}^2 \)-regular domain in \( \mathbb{R}^3 \) with a connected complement \( \mathbb{R}^3 \setminus \overline{D} \). Let \( \kappa_b \) and \( \rho_b \) represent the bulk modulus and density of the gas within each acoustic bubble, respectively. Additionally, let \( \kappa_c \) and \( \rho_c \) be real constants, with \( \kappa_c, \rho_c \in \mathbb{R}^+ \), describing the fixed properties of the homogeneous background medium \( \mathbb{R}^3 \setminus \overline{D} \). We define
\[
\kappa := \kappa_b \chi(D) + \kappa_c \chi(\mathbb{R}^3 \setminus \overline{D}) \quad \text{and} \quad \rho := \rho_b \chi(D) + \rho_c \chi(\mathbb{R}^3 \setminus \overline{D}),
\]
where \( \chi \) denotes the characteristic function of a domain. The homogeneous mathematical model describing the propagation of the total wave field \( u := u^{\text{sc}} + u^{\text{in}} \) is then formulated by the following transmission problem governed by the hyperbolic equation:
    \begin{align}\label{hyperbolic-problem}
        \begin{cases} 
            \kappa^{-1}(x) u_{tt} - \nabla \cdot \left( \rho^{-1}(x) \nabla u \right) = 0 & \text{in } (\mathbb{R}^{3} \setminus \partial D) \times (0, T), \\
            u \big|_{+} = u \big|_{-} & \text{on } \partial D, \\
            \rho_c^{-1} \partial_\nu u \big|_{+} = \rho_b^{-1} \partial_\nu u \big|_{-} & \text{on } \partial D, \\
            u(x, 0) = u_t(x, 0) = 0 & \text{for } x \in \mathbb{R}^3,
        \end{cases}
    \end{align}
where \( u^{\textit{in}} \) denotes the incident wave field generated by a point source located at \( x_0 \in \mathbb{R}^3 \setminus \overline{D} \), given by
    \begin{align}
          u^{\textit{in}}(x, t, x_0) := \rho_c\frac{\lambda\left(t - c_0^{-1} \vert x - x_0 \vert\right)}{\vert x - x_0 \vert},
    \end{align}
with \( \lambda \in \mathcal{C}^9(\mathbb{R}) \)\footnote{The reason behind taking this time-regularity of $\lambda$ was discussed in \cite{Arpan-Sini-SIMA}.} assumed to be a modulated causal temporal signal. Observe that $u^\textit{in}$ satisfies $k_c^{-1}u^\textit{in}_{tt}-\rho_c^{-1}\Delta u^\textit{in} = \lambda(t)\delta_{x_0}(x).$ The term \( u^{\textit{sc}} \) represents the scattered wave field.
The acoustic problem above is related to the linearized version of the nonlinear model for acoustic propagation 
 in bubbly media, see \cite{C-M-P-T-1, C-M-P-T-2}.
\par\noindent
The objective of this work is to derive an effective medium theory for the problem outlined above. We consider a cluster of bubbles, represented by
\[
D := \bigcup_{m=1}^M \bigcup_{l=1}^{\lfloor K(\cdot) + 1 \rfloor} D_{m_l},
\]
where \(\lfloor  K(\cdot) + 1 \rfloor\) denotes the floor function for a given non-negative, real-valued continuous function $K$. Each bubble \( D_{m_l} \) is defined as \( D_{m_l} = z_{m_l} + \varepsilon B_{m_l} \) with \( \varepsilon \ll 1 \), where \( B_{m_l} \) is centered at the origin and has a volume \(\text{Vol}(B_{m_l}) \sim 1\). Here, \( z_{m_l} \) represents the positions of the bubbles. The term \(\lfloor K(\cdot) + 1 \rfloor\) will be discussed shortly with further details.\\
\noindent
We now introduce the following two parameters
\[
c_0 := \sqrt{\frac{\rho_c}{\kappa_c}} \quad \text{and} \quad c_b := \sqrt{\frac{\rho_b}{\kappa_b}},
\]
which denote the wave speeds inside each bubble and within the background medium, respectively. Both parameters are assumed to be of order 1, implying that the contrast between wave speeds inside and outside the bubbles is not significant. However, a significant contrast in density and bulk modulus is assumed, leading to the following scaling properties for these two parameters:
    \begin{align}
          \kappa_b = \varepsilon^2 \overline{\kappa}_b \quad \text{and} \quad \rho_b = \varepsilon^2 \overline{\rho}_b, \quad \varepsilon \ll 1,
    \end{align}
where \(\overline{\kappa}_b\) and \(\rho_b\) are pre-factors independent of \(\varepsilon\). We further assume that the parameters associated with the homogeneous medium, \(\kappa_c\) and \(\rho_c\), are of order 1. 
\bigskip

\noindent
\textbf{\textit{Definitions.}} Let \( D \) denote the union of bubbles, expressed as \( D := \bigcup\limits_{m=1}^M \bigcup\limits_{l=1}^{\lfloor K(\cdot) + 1 \rfloor} D_{m_l} \). We define the following parameters:
\begin{enumerate}
    \item The parameter \(\varepsilon\) as the maximum diameter among all distributed bubbles, given by
   \[
   \varepsilon := \max\limits_{\substack{1 \leq m \leq M \\ 1 \leq l \leq \lfloor K(\cdot) + 1 \rfloor}} \textit{diam}(D_{m_l}).
   \]
   \item The parameter \(d\) as the minimum distance between any two distributed bubbles, defined as
   \[
   d := \min\limits_{1 \leq m, n \leq M} \min\limits_{\substack{1 \leq l, j \leq \lfloor K(\cdot) + 1 \rfloor \\ m_l \ne n_j}} \textit{dist}(D_{m_l}, D_{n_j}).
   \]
\end{enumerate}

\noindent \textbf{\textit{Assumption 1.}}\label{as2} Let \( \mathbf{\Gamma} \) denote a \( \mathcal{C}^2 \)-regular surface with unit area, such that either:  
\begin{enumerate}
    \item \( \mathbf{\Gamma} = \partial \Omega, \) or  
    \item \( \mathbf{\Gamma} \) is an open subset of \( \partial \Omega \), where \( \partial \Omega \) is the boundary of a \( \mathcal{C}^2 \)-regular, open, connected, and bounded subset $\Omega$ of \( \mathbb{R}^3 \).
\end{enumerate} 
The orientation of the normal on $\bm \Gamma$ inherits the exterior orientation of the one on $\partial \Omega.$
\vspace{0.05cm}

\noindent
Let \( K : \mathbf{\Gamma} \to \mathbb{R} \) be a given non-negative, continuous function. The domain \( \mathbf{\Gamma} \) is partitioned into \( M \) subdomains \( \Gamma_m \), each with an area of precisely \( \textit{Area}(\Gamma_m) = d^2 \). For a fixed point \( z_{m_1} \), each subdomain \( \Gamma_m \) is further subdivided into \( \lfloor K(z_{m_1}) + 1 \rfloor \) smaller subdomains \( \Gamma_{m_l} \), where \( z_{m_1} \in D_{m_l} \cap \Gamma_{m_l} \subseteq \Gamma_m \) for \( l = 1, 2, \ldots, \lfloor K(z_{m_1}) + 1 \rfloor \). Consequently, the number of gas bubbles within each subdomain \( \Gamma_m \) is determined to be \( \lfloor K(z_{m_1}) + 1 \rfloor \) and remains fixed.
\bigskip

\noindent
As the function $K(\cdot)$ is continuous then the induced function \( \lfloor K(\cdot) + 1 \rfloor \) is measurable and bounded, i.e. it belongs to $L^\infty(\Gamma)$, since $\lfloor K(\cdot) + 1 \rfloor^{-1}(n)=K^{-1}(n-1 \leq t<n)$, for any integer $n$, which is a measurable set since $K(\cdot)$ is continuous. The property $ \lfloor K(\cdot) + 1 \rfloor \in L^\infty(\Gamma)$ will be needed. \footnote{It is enough to have a pointwisely well defined function $K$, on $\bf \Gamma$, which is measurable.}

\noindent

\bigskip

\noindent
In this paper, we focus on the following regimes to model the cluster distributed in a two-dimensional bounded smooth surface
\begin{align} \label{regimes}
    M \sim d^{-2} \quad \text{and} \quad d \sim \varepsilon^{\frac{1}{2}}, \quad \varepsilon \ll 1.
\end{align}


    \subsection{Function Spaces.}       

\noindent
For the mathematical analysis, we first recall the functional decomposition (see \cite{D-L, raveski} for further details):
\begin{align} \label{decomposition-introduction}
\big(\mathrm{L}^2(D)\big)^3 = \mathbb{H}_{0}(\textit{div}\, 0, D) \oplus \mathbb{H}_{0}(\textit{curl}\, 0, D) \oplus \nabla \mathbb{H}_{\textit{arm}},
\end{align}
where
\begin{align}
    \begin{cases}
        \mathbb{H}_{0}(\textit{div}\, 0, D) = \Big\{ u \in \mathbb{H}(\textit{div}, D): \textit{div}\; u = 0 \; \text{in}\; D, \; u \cdot \nu = 0 \; \text{on}\; \partial D \Big\}, \\
        \mathbb{H}_{0}(\textit{curl}\, 0, D) = \Big\{ u \in \mathbb{H}(\textit{curl}, D): \textit{curl}\; u = 0 \; \text{in}\; D, \; u \times \nu = 0 \; \text{on}\; \partial D \Big\}, \\
        \nabla \mathbb{H}_{\textit{arm}} = \Big\{ u \in \big(\mathbb{L}^2(D)\big)^3: \exists \, \varphi \ \text{s.t.}\; u = \nabla \varphi, \; \varphi \in \mathbb{H}^1(D), \; \Delta \varphi = 0 \Big\}.
    \end{cases}
\end{align}
The Magnetization operator is defined by
\[
\bm{\mathbb{M}}^{(0)}_{\mathrm{B}_{m_l}}[f](x) := \nabla \int_{\mathrm{B}_{m_l}} \nabla_{y} \frac{1}{4 \pi | x - y |} \cdot f(y) \, dy.
\]
It is well known (see, for instance, \cite{friedmanI}) that the Magnetization operator \(\mathbb{M}^{(0)}_{\mathrm{B}_{m_l}}: \nabla \mathbb{H}_{\textit{arm}} \rightarrow \nabla \mathbb{H}_{\textit{arm}}\) induces a complete orthonormal basis, denoted \(\big(\lambda^{(3)}_{\mathrm{n}_{m_l}}, \mathrm{e}^{(3)}_{\mathrm{n}_{m_l}}\big)_{\mathrm{n} \in \mathbb{N}}\). We set \(\lambda_1^{(3)} := \min_{m_l} \max_n \lambda_{n_{m_l}}^{(3)}\).
We also define the real-valued Sobolev space \(\mathrm{H}_0^\mathrm{r}(0, \mathrm{T})\) for \(\mathrm{r} \in \mathbb{R}\) and \(\mathrm{T} \in (0, \infty]\) as follows:
\[
\mathrm{H}_0^\mathrm{r}(0, \mathrm{T}) := \Big\{ \mathrm{g}|_{(0, \mathrm{T})}: \mathrm{g} \in \mathrm{H}^\mathrm{r}(\mathbb{R}) \; \text{and} \; \mathrm{g}_{(-\infty, 0)} \equiv 0 \Big\}.
\]
A similar concept applies for functions that take values in a Hilbert space \(\mathrm{X}\), denoted \(\mathrm{H}^\mathrm{r}_0(0, \mathrm{T}; \mathrm{X})\). For \(\sigma > 0\) and \(\mathrm{r} \in \mathbb{Z}_+\), we define
    \begin{align}
          \mathrm{H}^\mathrm{r}_{0, \sigma}(0, \mathrm{T}; \mathrm{X}) := \Big\{ \mathrm{f} \in \mathrm{H}^\mathrm{r}_0(0, \mathrm{T}; \mathrm{X}): \sum_{\mathrm{n}=0}^\mathrm{r} \int_0^\mathrm{T} e^{-2\sigma \mathrm{t}} \| \partial_\mathrm{t}^\mathrm{n} \mathrm{f}(\cdot, \mathrm{t}) \|_\mathrm{X} \, d\mathrm{t} < \infty \Big\}.
    \end{align}
Before presenting the main result of this work, we introduce the parameter \( \mathbf{C}_{j_i} := \overline{\mathbf{C}}_{j_i} \varepsilon \), where \( \overline{\mathbf{C}}_{j_i} := \text{vol}(B_{j_i}) \frac{\rho_c}{\overline{\kappa}_{\mathrm{b}_{j_i}}} \).  

\noindent
We assume that the shape of each bubble \( D_{m_l} \) (for \( m = 1, 2, \ldots, M \)) is identical. Consequently, this implies that \( \mathbf{C}_{m_l} = \mathbf{C}_{j_i} \) for \( j = 1, 2, \ldots, M \) and \( i,l= 1, 2, \ldots, \lfloor K(\cdot) + 1 \rfloor \). Furthermore, let us define \( \overline{\mathbf{C}} := \overline{\mathbf{C}}_{j_i} \) as the scaled value of \( \mathbf{C}_{j_i} \).  


    \subsection{Statement of the result}    
    
\begin{theorem}\label{mainth}    

Let \( \bm{\Gamma} \) and the acoustic bubbles satisfy  Assumption \textcolor{blue}{1}. Then, under the following additional condition with \(K_\text{max} := \sup_{z_{m_l}}\big(K(z_{m_l}) + 1\big)\),
    \begin{equation}\label{condition}
          \sqrt{K_\text{max}} \sum_{\substack{j=1 \\ j \neq m}}^M \sum_{\substack{i,l=1 \\ i \neq l}}^{\lfloor K(\cdot) + 1 \rfloor} \frac{\mathbf{C}_{j_i}}{4\pi |z_{m_l} - z_{j_i}|} < \omega_M^2,
    \end{equation}
we have the following asymptotic expansion for \((x,t) \in \big(\mathbb{R}^3 \setminus \overline{\mathbf{\Omega}}\big) \times (0,T)\):
    \begin{align}
          u^\textit{sc}(x,t) - \mathbcal{W}^\textit{sc}(x,t) = \mathcal{O}(\varepsilon^\frac{1}{2}) \; \text{as} \; \varepsilon \to 0,
    \end{align}
where \(\mathbcal{W}(x,t) = \mathbcal{W}^\textit{sc} + u^\textbf{in}\) is the solution of the following dispersive acoustic model:
    \begin{equation} \label{pde-intro}
        \begin{cases}
            \big(c_0^{-2} \partial_t^2 - \Delta\big)\mathbcal{W}= \lambda(t)\delta_{x_0}(x), & \text{in}\ \mathbb{R}^3 \setminus \mathbf{\Gamma} \times (0,T), \\
             \mathbf{\big[} \mathbcal{W}\mathbf{\big]} = 0, & \text{on}\ \mathbf{\Gamma}\times(0,T), \\

             \mathbf{\Big[} \partial_\nu \mathbcal{W}\mathbf{\Big]} - \displaystyle \overline{\bm{C}}\ \big(\lfloor K(\cdot) + 1 \rfloor\big)\omega_M^{-1}\int_0^t \sin\big(\omega_M^{-1}(t-\tau)\big)\, \partial_\tau^2\mathbcal{W}(x,\tau)\, d\tau=0 & \text{on}\ \mathbf{\Gamma}\times(0,T), \\
             \partial_t^i \mathbcal{W}(x,0) = 0, & \text{in}\ \mathbb{R}^3,\; i = 0,1.
        \end{cases}
    \end{equation}
    
\end{theorem}        

\noindent 
Interface boundary conditions of the form (\ref{pde-intro}), related to the existence of dispersion, appear naturally in applied acoustics, see \cite{D-B, R-H} where a more general family of impedance type boundary conditions of this form is studied with general kernels $K(\cdot)$. In our case $K$ is the sine function, i.e. $K(t)=sin(\omega_M^{-1} t)$. Such kernels need to satisfy admissibility conditions as the reality, passivity, and causality, see \cite{D-B, R-H}. Such conditions are satisfied in our case. Let us also observe the following two extreme cases:
\begin{enumerate}
    \item When $\omega_M \to \infty,$ then the transmission condition (\ref{pde-intro}) becomes $\mathbf{\Big[} \partial_\nu \mathbcal{W}_\infty\mathbf{\Big]}=0.$ In this case, we have $\mathbcal{W}^\textit{sc}_\infty = - u^\textit{in}.$ Therefore, $\bm{\Gamma}$ becomes a fully transparent screen.
    \item  When $\omega_M \to 0,$ then the transmission condition (\ref{pde-intro}) becomes $\mathbf{\Big[} \partial_\nu \mathbcal{W}_0\mathbf{\Big]} - \overline{\bm{C}}\ \big(\lfloor K(\cdot) + 1 \rfloor\big) \partial^2_t\mathbcal{W}_0 =0.$ In this case, we retrieve the non-dispersive transmission condition and hence $\bm{\Gamma}$ becomes a reflective/transmissive screen. If in addition, $\overline{\bm{C}}$ or $K$ are large then $\Gamma$ would behave as a fully reflective screen as in this case, at the limit, $\mathbcal{W}_0 =0$.
    \item For a moderate resonance $\omega_M$, the interface behaves as a screen with a memory effect which reflects the dispersion generated by the resonance.
\end{enumerate}


    \subsection{Comparison with Existing Mathematical and Physical Literature} 

\noindent
The study of effective transmission conditions for clusters of acoustic subwavelength resonators aligns closely with research in wave propagation, Metamaterials, and scattering theory. However, the specific contributions of this work can be distinguished in several key areas:
\begin{enumerate}
      \item \textit{Asymptotic Analysis of Subwavelength Resonators}.
       This study builds on foundational mathematical works that employ asymptotic techniques to model subwavelength resonators, particularly in acoustic and electromagnetic media. Classic analyses, such as those by Foldy \cite{Foldy} and Keller \cite{Keller}, have investigated wave scattering by discrete scatterers, often using multipole expansions or homogenization techniques. A comprehensive account on these approaches, and more classical ones, can be found in the books by Martin \cite{Martin-1, Martin-2}. More recently, works such as those by Ammari and co-authors have advanced the understanding of resonance phenomena by deriving precise asymptotics for subwavelength resonators in various geometric configurations, see \cite{AZ18, Ammari-et-al-0, Ammari-et-al-2, Ammari-et-al-3}, but mainly flat surfaces and in the time-harmonic regime in general.
       In comparison, the present study extends these methods by considering a cluster of resonators supported on smooth but non-flat hypersurfaces making the transmission condition more intricate and physically realistic. This distinguishes it from the mentioned studies focused on flat surfaces or unbounded configurations. The closest reference to our work is \cite{habib-sini} where the authors characterized the transmission condition in the time-harmonic regime for non-flat surfaces.  Finally, the current work complement the previous work \cite{Arpan-Sini-arxiv} where we have characterized the effective medium when the resonators are distributed in volumetric smooth and bounded domains. In that work, the effective model is given by compactly supported potential with memory effect (which translates the resonant effect as well). 

       \item \textit{Acoustic Metamaterials and Resonant Screens}.
       In the field of acoustic metamaterials, subwavelength resonators have been utilized to design devices with extraordinary properties, such as negative refractive indices, perfect absorbers, and cloaking devices. Existing studies in this domain often focus on homogenized models or numerical simulations to understand wave manipulation. Notably, the works of Cummer et al. \cite{Cummer-Christensen-Alu} and Christensen and García de Abajo \cite{Christensen-de Abajo}, and the references therein, have explored how bubbles and other resonant elements create effective acoustic metamaterials.
       The present work complements these studies by offering a rigorous mathematical characterization of a resonant screen, providing a clear physical interpretation through the Dirac-like surface potential. This bridges the gap between mathematical rigor and practical application, providing insights into the behavior of reflective and transmissive surfaces in regimes not fully explored in prior metamaterial studies.
       
       \item  \textit{Memory Effects and Time-Dependent Wave Dynamics}.
       The incorporation of memory effects into the transmission condition places this study in the context of dispersive and nonlocal wave equations. Memory effects are increasingly studied in models of viscoelasticity, electromagnetic waves, and acoustic scattering, as seen in the works of Mainardi \cite{Mainardi} and others. In contrast to these fields, where memory effects often arise from intrinsic material properties, with polynomials kernels, this study attributes them to the geometric arrangement and resonance of the bubbles, hence with sinusoidal kernels.
       The resulting time-convolution formulation offers a new perspective on designing surfaces that exhibit tunable memory properties. This feature can be particularly relevant for applications involving energy harvesting, acoustic insulation, or advanced wave control technologies.
\end{enumerate}


    \section{Proof of Theorem \ref{mainth}}     

\noindent
First, let us  highlight the key difference between our earlier work \cite{Arpan-Sini-arxiv} and the current one. In the previous one, the focus was on the volumetric distribution of acoustic bubbles, whereas in the current one, we deal with distribution of the bubbles along a surface (open or closed). Consequently, here we work with a surface integral equation rather than a volumetric integral equation.
\vspace{0.1cm}

\noindent 
We divide this section into three steps. As a first step, we recall the close form of the acoustic field generated by the cluster of resonating bubbles, from \cite{Arpan-Sini-JEEQ}, and rewrite it according to the distribution of these bubbles along the given surface. This allows us to naturally motivate the associated effective model, which exhibits a dispersive behavior. As a second step, we study the well-posedness, and the smoothing property, of the surface integral equation corresponding to this dispersive effective model. As a third step, we provide a detailed justification of the stated Theorem \ref{mainth} by comparing the close form solution of the cluster of bubbles with the discretized form of the solution the effective integral solution.


    \subsection{The dispersive effective model and the corresponding surface integral equation}     

\noindent
In order to deduce the dispersive effective mode and the corresponding surface integral equation, we first state the following proposition.
\begin{proposition}\cite[Theorem 1.1]{Arpan-Sini-JEEQ}\label{main-prop}
    Consider the acoustic problem (\ref{hyperbolic-problem}) generated by a cluster of resonating acoustic gas bubbles \( D_j \) for \( j = 1, 2, \ldots, M \). Then, under the following conditions:
\begin{align}\label{inversion-cond}
    \frac{\rho_\mathrm{c}}{4\pi} \, \text{vol}(\mathrm{B}_j) \, \Big(\frac{\varepsilon}{d}\Big)^6 \, \Big(\frac{1}{\lambda_1^{(3)}}\Big)^2 < 1 
    \quad \text{and} \quad 
    \mathbf{C} \, \max_{1 \leq m \leq M} \sum_{\substack{j=1 \\ j \neq m}}^M \frac{1}{4\pi |z_m - z_j|} < \omega_M^2,
\end{align}
with \( C := \max\limits_{1 \leq j \leq M} \mathbf{C}_j \), the scattered field \( u^\textit{sc}(\mathrm{x}, \mathrm{t}) \) has the following asymptotic expansion:
\begin{align} \label{assymptotic-expansion-us}
    u^\textit{sc}(\mathrm{x}, \mathrm{t}) = -\sum_{m=1}^M \frac{\mathbf{C}_j}{4\pi |\mathrm{x} - \mathrm{z}_m|} \mathrm{Y}_m\big(\mathrm{t} - \mathrm{c}_0^{-1} |\mathrm{x} - \mathrm{z}_m|\big) + \mathcal{O}(M\varepsilon^{2}) \quad \text{as} \quad \varepsilon \to 0,
\end{align}
for \( (\mathrm{x}, \mathrm{t}) \in \mathbb{R}^3 \setminus \mathrm{P} \times (0, \mathrm{T}) \) with \( \overline{\mathbf{\Omega}} \subset \subset \mathrm{P} \), where \( \big(\mathrm{Y}_j\big)_{j=1}^M \) is the vector solution to the following non-homogeneous second-order matrix differential equation with zero initial conditions:
\begin{align} \label{matrixmulti}
    \begin{cases}
        \omega_M^2 \frac{\mathrm{d}^2}{\mathrm{d} \mathrm{t}^2} \mathrm{Y}_m(\mathrm{t}) + \mathrm{Y}_m(\mathrm{t}) + \sum\limits_{\substack{j=1 \\ j \neq m}}^M \frac{\mathbf{C}_j}{4\pi |z_m - z_j|} \frac{\mathrm{d}^2}{\mathrm{d} \mathrm{t}^2} \mathrm{Y}_j\big(\mathrm{t} - \mathrm{c}_0^{-1} |\mathrm{z}_m - \mathrm{z}_j|\big) = \frac{\partial^2}{\partial t^2} u^\textit{in}, &\text{in} \; (0, \mathrm{T}), \\
        \mathrm{Y}_m(0) = \frac{\mathrm{d}}{\mathrm{d} \mathrm{t}} \mathrm{Y}_m(0) = 0,
    \end{cases}
\end{align}
where \( \omega_M^2 = \frac{\rho_c}{2 \overline{\kappa}_{\mathrm{b}_m}} \mathrm{A}_{\partial B_m} \) represents the square of the Minnaert frequency \( \omega_M \) of the bubble \( D_m \) (see, for instance, \cite{AZ18, DGS-21}), and \( \mathbf{C}_j := \overline{\mathbf{C}}_j \varepsilon \) with \( \overline{\mathbf{C}}_j := \text{vol}(B_j) \frac{\rho_c}{\overline{\kappa}_{\mathrm{b}_j}} \).
\\
Here, \( \displaystyle \mathrm{A}_{\partial B_m} := \frac{1}{|\partial B_m|} \int_{\partial B_m} \int_{\partial B_m} \frac{(\mathrm{x} - \mathrm{y}) \cdot \nu_\mathrm{x}}{|\mathrm{x} - \mathrm{y}|} \, d\sigma_\mathrm{x} \, d\sigma_\mathrm{y} \) is a geometric constant. The well-posedness of the system of differential equations (\ref{matrixmulti}) is discussed in \cite[Section 2.4]{Arpan-Sini-JEEQ}.
\end{proposition}
\noindent
Here, we aim to derive the effective medium properties for a configuration where subwavelength resonators, specifically acoustic gas bubbles, are distributed in a locally non-periodic manner over a \(C^2\)-regular surface \(\bm{\Gamma}\) of unit area. To satisfy \textit{Assumption \textcolor{blue}{1}} concerning the distribution of bubbles, a key step involves reformulating the algebraic system (\ref{matrixmulti}) into a generalized algebraic system that reflects the local distribution characteristics of the bubbles. This reformulation is analogous to the derivation presented in \cite[Section 4.1]{Arpan-Sini-arxiv}, and we omit detailed steps here for brevity. Accordingly, with a slight abuse of notation, we reformulate the algebraic system (\ref{matrixmulti}) as follows:
\begin{align}\label{tran1}
    \begin{cases}
             \mathbb{A}\frac{\mathrm{d}^2}{\mathrm{d}\mathrm{t}^2}\bm{\mathbb{Y}}_m(\mathrm{t}) + \bm{\mathbb{Y}}_m(\mathrm{t}) + \sum\limits_{\substack{j=1 \\ j\neq m}}^{[d^{-2}]}\mathbb C_{mj} \cdot\frac{\mathrm{d}^2}{\mathrm{d}\mathrm{t}^2}\mathbb Y_j(\mathrm{t}-\mathrm{c}_0^{-1}|\mathrm{z}_{m}-\mathrm{z}_{j}|) = \mathbcal{H}_m^\textit{in}\; \mbox{ in } (0, \mathrm{T}),
             \\ \bm{\mathbb{Y}}_m(\mathrm{0}) = \frac{\mathrm{d}}{\mathrm{d}\mathrm{t}}\bm{\mathbb{Y}}_m(\mathrm{0}) = 0,   
    \end{cases}
\end{align}
where 
    \begin{align}
        \mathbb C_{mj}:=
            \begin{pmatrix}
                   0 & \frac{\mathbf{C}_{j_2}}{4\pi |z_{m_1} - z_{j_2}|} & \ldots & \frac{\mathbf{C}_{j_{\lfloor K(\cdot) + 1 \rfloor}}}{4\pi |z_{m_1} - z_{j_{\lfloor K(\cdot) + 1 \rfloor}}|} \\
                  \frac{\mathbf{C}_{j_1}}{4\pi |z_{m_2} - z_{j_1}|} & 0 & \ldots & \frac{\mathbf{C}_{j_{\lfloor K(\cdot) + 1 \rfloor}}}{4\pi |z_{m_2} - z_{j_{\lfloor K(\cdot) + 1 \rfloor}}|} \\
                  \vdots & \vdots & \ddots & \vdots \\
                  \frac{\mathbf{C}_{j_1}}{4\pi |z_{m_{\lfloor K(\cdot) + 1 \rfloor}}-z_{j_1}|} & \frac{\mathbf{C}_{j_2}}{4\pi |z_{m_{\lfloor K(\cdot) + 1 \rfloor}}-z_{j_2}|} & \ldots & 0
            \end{pmatrix}
    \end{align}
is describing the $\big({\lfloor K(\cdot) + 1 \rfloor}\big)^2$-block interactions between the inclusions located in $\Omega_m$ and $\mathrm{D}_j$ for $j\ne m$ and the incident source $\mathbcal H^\textbf{in}_m := \Big(\frac{\partial^2}{\partial t^2}u^\textbf{in}(z_{m_1}), \frac{\partial^2}{\partial t^2}u^\textbf{in}(z_{m_2}), \ldots, \frac{\partial^2}{\partial t^2}u^\textbf{in}(z_{m_l})\Big)^t$. We define $\mathbb{A}$ as
\begin{align}
    \mathbb{A} = \big[\mathbcal a_{ij}\big]_{i,j=1}^{\lfloor K(\cdot) + 1 \rfloor}:=
     \begin{pmatrix}
        \omega_M^2 & 0 & \dots & 0 \\
        0 & \omega_M^2 & \dots & 0 \\
        \vdots & \vdots & \ddots & \vdots \\
        0 & 0 & \dots & \omega_M^2 
    \end{pmatrix}, 
\end{align}
with $\bm{\mathbb{Y}}_m=\Big(\widetilde{Y}_{m_1},\widetilde{Y}_{m_2},\ldots,\widetilde{Y}_{m_{K+1}}\Big)^t.$ Then, the well-posedness of the above algebric system can be shown under the condition with \(K_\text{max} := \sup_{z_{m_l}}\big(K(z_{m_l}) + 1\big)\),
    \begin{equation}
          \sqrt{K_\text{max}} \sum_{\substack{j=1 \\ j \neq m}}^M \sum_{\substack{i,l=1 \\ i \neq l}}^{\lfloor K(\cdot) + 1 \rfloor} \frac{\mathbf{C}_{j_i}}{4\pi |z_{m_l} - z_{j_i}|} < \omega_M^2.
    \end{equation}
As a next step, we aim to connect the previously discussed algebraic system (\ref{tran1}) to an general surface integral equation. Specifically, we analyze the following integral equation for $(\mathrm{x},t) \in \mathbb{R}^3\times (0, T)$:
    \begin{align}\label{surface-integral-equn}
          \rchi_{\mathbf{\Gamma}} \, \mathbb{A} \cdot \frac{\partial^2}{\partial t^2} \mathbcal{F}_m (\mathrm{x}, \mathrm{t}) + \mathbcal{F}_m (\mathrm{x}, \mathrm{t}) + \int_{\mathbf{\Gamma}} \mathbcal{C}(x, y) \cdot \frac{\partial^2}{\partial t^2} \mathbcal{F}_m(y, t - c_0^{-1} \vert x - y \vert) \, dy = \frac{\partial^2}{\partial t^2} \mathbcal{F}_m^\textbf{in}(\mathrm{x}, \mathrm{t}),
    \end{align}
where \(\mathbcal{F}_m^\textbf{in} = \big(u^\textbf{in}, u^\textbf{in}, \ldots, u^\textbf{in}\big)^t\). Initial conditions for \(\mathbcal{F} := \big(\mathcal{F}_{m_1}, \mathcal{F}_{m_2}, \ldots, \mathcal{F}_{m_{K+1}}\big)^t\) will be incorporated to the first order in this framework. The term \(\mathbcal{C}\) is defined as the interaction matrix representation given by
    \begin{align}
        \mathbcal{C}(x, y) = 
            \begin{pmatrix}
                0 & \frac{\overline{\bm{C}}}{4\pi|x-y|} & \cdots & \frac{\overline{\bm{C}}}{4\pi|x-y|} \\
                \frac{\overline{\bm{C}}}{4\pi|x-y|} & 0 & \cdots & \frac{\overline{\bm{C}}}{4\pi|x-y|} \\
                \vdots & \vdots & \ddots & \vdots \\
                \frac{\overline{\bm{C}}}{4\pi|x-y|} & \frac{\overline{\bm{C}}}{4\pi|x-y|} & \cdots & 0 
            \end{pmatrix},\; \text{for}\; x \neq y.
    \end{align}
Additionally, we define
    \begin{align}
        \mathbb{A} = \big[\mathbcal{a}_{ij}\big]_{i,j=1}^{\lfloor K(\cdot) + 1 \rfloor}:= 
        \begin{pmatrix}
              \omega_M^2 & 0 & \dots & 0 \\
              0 & \omega_M^2 & \dots & 0 \\
              \vdots & \vdots & \ddots & \vdots \\
              0 & 0 & \dots & \omega_M^2
        \end{pmatrix}.
    \end{align}
Upon analyzing the interaction matrix \(\mathbcal{C}\) and the definition of \(\mathbcal{F}_m\), the system of integral equations (\ref{surface-integral-equn}) can be reformulated as follows for \((x,t) \in \mathbb{R}^3\times(0, T)\)
    \begin{align}\label{ef-eq-1}
         \rchi_{\mathbf{\Gamma}} \, \mathbb{A} \cdot \frac{\partial^2}{\partial t^2} \mathbcal{F}_m(x, t) + \mathbcal{F}_m(\mathrm{x}, \mathrm{t}) + \int_{\mathbf{\Gamma}} 
         \begin{pmatrix}
               \frac{\overline{\bm{C}}}{4\pi|x-y|} \cdot \sum\limits_{l=1}^{\lfloor K(\cdot) + 1 \rfloor} \mathcal{F}_{m_l}(\mathrm{y}, \mathrm{t} - c_0^{-1} \vert x - y \vert) \\
               \vdots \\
               \frac{\overline{\bm{C}}}{4\pi|x-y|} \cdot \sum\limits_{l=1}^{\lfloor K(\cdot) + 1 \rfloor} \mathcal{F}_{m_l}(\mathrm{y}, \mathrm{t} - c_0^{-1} \vert x - y \vert)
        \end{pmatrix} 
        \, dy = \frac{\partial^2}{\partial t^2} \mathbcal{F}_m^\textit{in}.
    \end{align}
This formulation reveals that \(\mathbcal{F}_m(\cdot, \cdot)\) can be represented as the summation of its individual components, specifically \(\sum\limits_{\ell=1}^{\lfloor K(\cdot) + 1 \rfloor} \mathcal{F}_{m_\ell}(\cdot, \cdot)\). By summing terms in the previous expression, we obtain the corresponding integral equation for \(\sum\limits_{\ell=1}^{\lfloor K(\cdot) + 1 \rfloor} \mathcal{F}_{m_\ell}(\cdot, \cdot)\), applicable for \((x,t) \in \mathbb{R}^3\times(0, T)\):
    \begin{align}\label{ef-eq-2}
          &\nonumber\rchi_{\mathbf{\Gamma}} \, \omega_M^2 \, \frac{\partial^2}{\partial t^2} \sum_{l=1}^{\lfloor K(\cdot) 
          + 1 \rfloor} \mathcal{F}_{m_l}(x, t) + \sum_{l=1}^{\lfloor K(\cdot) + 1 \rfloor} \mathcal{F}_{m_l}(x, t) 
          \\ &+ \int_{\mathbf{\Gamma}} (\lfloor K(\cdot) + 1 \rfloor) \, \frac{\overline{\bm{C}}}{4\pi|x-y|}  \, \frac{\partial^2}{\partial t^2} \sum_{l=1}^{\lfloor K(\cdot) + 1 \rfloor} \mathcal{F}_{m_l}(\mathrm{y}, \mathrm{t} - c_0^{-1} \vert x - y \vert) \, dy 
          = (\lfloor K(\cdot) + 1 \rfloor) \frac{\partial^2}{\partial t^2} u^\textit{in}(\mathrm{x}, \mathrm{t}).
    \end{align}
By redefining the unknown function as \( \bm{\mathrm{Y}} := \frac{1}{\lfloor K(\cdot) + 1 \rfloor} \sum\limits_{l=1}^{\lfloor K(\cdot) + 1 \rfloor} \mathcal{F}_{m_l}(x, t)\), we arrive at the final integral equation for $(x,t)\in \mathbb{R}^3\times(0, \mathrm{T})$:
\begin{equation}\label{effective-equation}
\omega_M^2\; \rchi_{\mathbf{\Gamma}}\; \frac{\partial ^2}{\partial t^2} \bm{\mathrm{Y}} (\mathrm{x},\mathrm{t}) + \bm{\mathrm{Y}} (\mathrm{x},\mathrm{t}) + \int_{\mathbf{\Gamma}}(\lfloor K(\cdot) + 1 \rfloor)\frac{\overline{\bm{C}}}{4\pi|x-y|} \frac{\partial ^2}{\partial t^2}{\mathbf{Y} (y, \mathrm{t}-c_0^{-1}\vert \mathrm{x}-\mathrm{y}\vert)} d\mathrm{y} = \frac{\partial ^2}{\partial t^2}u^\textit{in}(\mathrm{x},\mathrm{t}).
\end{equation}
To this integro-differential equation, we impose zero initial conditions for $\mathbf{Y}$ and its first time derivative.
\\
We show the unique solvability of the above Lippmann-Schwinger equation (\ref{effective-equation}) in $\mathrm{H}^{\mathrm{r}}_{0,\sigma}\big(0,\mathrm{T};\mathrm{L}^2(\mathbf{\Gamma})\big)$ for $u^\textbf{in}$ belonging to the space $\mathrm{H}^{\mathrm{r}+1}_{0,\sigma}\big(0,\mathrm{T};\mathrm{L}^2(\mathbf{\Gamma})\big)$. In addition, we show some regularity results corresponding to the solution of (\ref{effective-equation}) in Section \ref{seclipp}.
\\
We introduce the unknown variable \( \mathbf{U} \) through the relation \( \mathbf{Y} = \frac{\partial^2}{\partial t^2}\mathbf{U} \), where \( \mathbf{U} \) satisfies the following Lippmann-Schwinger equation for \( (x, t) \in \mathbb{R}^3 \times (0, \mathrm{T}) \)
\begin{equation}\label{effective-equation-2}
\omega_M^2 \; \rchi_{\mathbf{\Gamma}} \; \frac{\partial^2}{\partial t^2} \mathbf{U} (x, t) + \mathbf{U} (x, t) + \int_{\mathbf{\Gamma}} (\lfloor K(\cdot) + 1 \rfloor)\frac{\overline{\bm{C}}}{4\pi|x-y|} \frac{\partial^2}{\partial t^2} \mathbf{U} (y, t-c_0^{-1} \vert x-y \vert) dy = u^\textit{in} (x, t),
\end{equation}
subject to zero initial conditions for \(\mathbf{U}\) up to the first order.
\\
We now define the following:
\begin{align}\label{effective-equation-3}
    \mathbcal{W}(x, t) := 
    \begin{cases}
        \mathbf{U}(x, t) + \omega_M^2 \; \frac{\partial^2}{\partial t^2} \mathbf{U} (x, t), & \text{if } (x, t) \in \mathbf{\Gamma} \times (0, T), \\[8pt]
        u^\textit{in}(x, t) - \displaystyle\int_{\mathbf{\Gamma}} (\lfloor K(\cdot) + 1 \rfloor)\frac{\overline{\bm{C}}}{4\pi|x-y|} \frac{\partial^2}{\partial t^2} \mathbf{U} (y, t-c_0^{-1} \vert x-y \vert) dy, & \text{for } (x, t) \in \mathbb{R}^3 \setminus \mathbf{\Gamma} \times (0, T).
    \end{cases}
\end{align}
The geometric configuration consists of a bounded domain \(\mathbf{\Omega}\) with a \(\mathcal{C}^2\)-smooth boundary \(\mathbf{\Gamma}\), and its complement, \(\mathbf{\Omega}^+ := \mathbb{R}^3 \setminus \overline{\mathbf{\Omega}}\), which is assumed to be connected. The Dirichlet trace of a function \(u\) on the boundary \(\mathbf{\Gamma}\), taken from the interior and exterior of \(\mathbf{\Omega}\), is denoted by \(u|_{-}\) and \(u|_{+}\), respectively. The corresponding Neumann traces, with the outward-pointing unit normal vector, are given by \(\partial_\nu u\big|_{-}\) for the interior and \(\partial_\nu u\big|_{+}\) for the exterior. The jumps of these quantities across the interface \(\mathbf{\Gamma}\) are expressed as
\(
[ u ] := u|_{+} - u|_{-},\; \text{and}\; \Big[ \partial_\nu u \Big] := \partial_\nu u\big|_{+} - \partial_\nu u\big|_{-}.
\)
The involvement of the surface integral induces a discontinuity across the boundary, leading to the following transmission conditions
    \begin{align}
         \mathbf{\Big[} \partial_\nu \mathbcal{W}\mathbf{\Big]} 
         = \overline{\bm{C}}\, \big(\lfloor K(\cdot) + 1 \rfloor\big) \partial_t^2\mathbf{\mathrm{U}}(x,t)
    \end{align}
which further imply the following conditions
    \begin{align*}
        \mathbf{\Big[} \partial_\nu \mathbcal{W}\mathbf{\Big]} + \displaystyle \overline{\bm{C}}\ \big(\lfloor K(\cdot) + 1 \rfloor\big)\Big(-\omega_M^{-2} \mathbcal{W}(x,t) +\omega_M^{-3}\int_0^t \sin\big(\omega_m^{-1}(t-\tau)\big)\, \mathbcal{W}(x,\tau)\, d\tau\Big) =0.
    \end{align*}
    An integration by parts shows that
    \begin{align*}
        \omega_M^{-2} f(x,t) - \omega_M^{-3}\int_0^t \sin\big(\omega_M^{-1}(t-\tau)\big)\, f(x,\tau)\, d\tau = \omega_M^{-1}\int_0^t \sin\big(\omega_M^{-1}(t-\tau)\big)\, \partial^2_\tau f(x,\tau)\, d\tau,
    \end{align*}
which further implies that the transmission condition above reduces to
    \begin{align}\label{Transmision}
         \mathbf{\Big[} \partial_\nu \mathbcal{W}\mathbf{\Big]} -\overline{\bm{C}}\ \big(\lfloor K(\cdot) + 1 \rfloor\big) \omega_M^{-1}\int_0^t \sin\big(\omega_M^{-1}(t-\tau)\big)\, \partial^2_\tau \mathbcal{W}(x,\tau)\, d\tau = 0.
    \end{align}
Consequently, $\mathbcal{W}$ satisfies the following hyperbolic problem:
    \begin{equation} \label{pde}
        \begin{cases}
            \big(c_0^{-2} \partial_t^2 - \Delta\big)\mathbcal{W}= \lambda(t)\delta_{x_0}(x), & \text{in}\ \mathbb{R}^3 \setminus \mathbf{\Gamma} \times (0,T), \\
            \mathbf{\big[} \mathbcal{W}\mathbf{\big]} = 0, & \text{on}\ \mathbf{\Gamma}\times(0,T), \\
            \mathbf{\Big[} \partial_\nu \mathbcal{W}\mathbf{\Big]} - \displaystyle \overline{\bm{C}}\ \big(\lfloor K(\cdot) + 1 \rfloor\big)\omega_M^{-1}\int_0^t \sin\big(\omega_M^{-1}(t-\tau)\big)\, \partial_\tau^2\mathbcal{W}(x,\tau)\, d\tau=0 & \text{on}\ \mathbf{\Gamma}\times(0,T), \\
             \partial_t^i \mathbcal{W}(x,0) = 0, & \text{in}\ \mathbb{R}^3,\; i = 0,1.
        \end{cases}
    \end{equation}
We note that $\mathbcal{W}$ is a solution of (\ref{pde}) if and only if $\mathbf{\mathrm{U}}$ satisfies the following surface integral for $(\mathrm{x},t) \in \mathbb{R}^3\times (0, T)$:
\begin{equation}
\omega_M^2 \; \rchi_{\mathbf{\Gamma}} \; \frac{\partial^2}{\partial t^2} \mathbf{U} (x, t) + \mathbf{U} (x, t) + \int_{\mathbf{\Gamma}} (\lfloor K(\cdot) + 1 \rfloor)\frac{\overline{\bm{C}}}{4\pi|x-y|} \frac{\partial^2}{\partial t^2} \mathbf{U} (y, t-c_0^{-1} \vert x-y \vert) dy = u^\textit{in} (x, t).
\end{equation}
Next, we proceed to the following Section to prove the solvability of the above surface integral equation.


    \subsection{Solvability of the Lippmann-Schwinger equation (\ref{effective-equation})}\label{seclipp}    

\noindent
To establish the solvability of the Lippmann-Schwinger equation (\ref{effective-equation_1}), we employ the framework introduced by Bamberger and Ha-Duong in \cite{hduong} and further developed by Sayas in \cite{sayas}, both of which utilize the Fourier-Laplace transform as a fundamental tool. In addition, we draw upon the methodologies and results presented in \cite{le-monk, Arpan-Sini-SIMA, sini-wang}, where this technique coupled with fine spectral decomposition of involved integral operators has been effectively applied to demonstrate the solvability of various classes of equations. Based on these approaches, we present and prove the following proposition.

\begin{proposition}\label{prop1}     

The Lippmann-Schwinger equation, stated on $ \mathbf{\Gamma}\times (0, T)$, with $\hbar:=\omega_M^2$
\begin{equation}\label{effective-equation_1}
\hbar\; \rchi_{\mathbf{\Gamma}}\; \frac{\partial ^2}{\partial t^2} \bm{\mathrm{Y}} (\mathrm{x},\mathrm{t}) + \bm{\mathrm{Y}} (\mathrm{x},\mathrm{t}) + \int_{\mathbf{\Gamma}}(\lfloor K(\cdot) + 1 \rfloor)\frac{\overline{\bm{C}}}{4\pi|x-y|} \frac{\partial ^2}{\partial t^2}{\mathbf{Y} (y, t-c_0^{-1}\vert x-y\vert)} dy = \frac{\partial ^2}{\partial t^2}u^\textit{in}(\mathrm{x},\mathrm{t}),
\end{equation}
has a unique solution in $\mathrm{H}^{\mathrm{r}}_{0,\sigma}\big(0,\mathrm{T};\mathrm{L}^2(\mathbf{\Gamma})\big)$ for $u^\textit{in} \in \mathrm{H}^{\mathrm{r}+1}_{0,\sigma}\big(0,\mathrm{T};\mathrm{L}^2(\mathbf{\Gamma})\big).$

\end{proposition}       

\begin{proof}         
Here, for simplicity, we present the proof for the case \( K \equiv 0 \). However, since \( \lfloor K(\cdot) + 1\rfloor \in L^\infty(\bm{\Omega}) \), a similar argument applies when \( K \) is a variable function. Next, we define the retarded single-layer wave operator as follows
\[
\mathbcal{S}u(x,t) := \int_{\mathbf{\Gamma}} \frac{1}{4\pi |x - y|} u(y, t - c_0^{-1}|x - y|) \, d\sigma_y.
\]
Next, we analyze this integral representation in the Fourier-Laplace domain by applying the Fourier-Laplace transform with respect to the time variable. For a Laplace-transformable function \(u\), with transform parameter \(\omega = \xi + i\sigma\), we define the Fourier-Laplace transform as
\[
\mathcal{L}(u)(\omega) = \hat{u}(\omega) = \int_{-\infty}^{+\infty} e^{i\omega t} u(t) \, dt = \mathcal{F}(e^{-\sigma t} u)(\xi),
\]
where \(\mathcal{F}\) denotes the Fourier transform.
\\
It is known that if \(u\) is causal, i.e., \(u(t) = 0\) for \(t < 0\), and for real \(\sigma\), \(e^{-\sigma t} u\) is tempered, then \(u\) is Laplace-transformable, and its Fourier transform \(\mathcal{F}\) is well-defined. This leads us to express the retarded wave operator in the Fourier-Laplace domain as:
\[
\mathcal{L}(\mathbcal{S}u)(x,\omega) = \int_{\mathbf{\Gamma}} \frac{e^{i\omega |x - y|}}{4\pi |x - y|} \hat{u}(y,\omega) \, d\sigma_y =: \hat{\mathbcal{S}}_\omega \hat{u}(x,\omega),
\]
which is the single-layer potential for the Helmholtz equation, valid for \(x \notin \mathbf{\Gamma}\) and for all frequencies \(\omega\).\\
Due to the assumption of causality for the function \( u \), in the Fourier-Laplace transforms, we are allowed to consider the frequency \(\omega\) in the upper complex plane, defined as \(\mathbb{C}_{\omega_0} := \{ \omega \in \mathbb{C} : \Im(\omega) \geq \omega_0 > 0 \}\). This ensures that the transform is well-defined in a region where the imaginary part of \(\omega\) is bounded below by a positive constant, \(\omega_0\), consistent with the temporal decay condition imposed by causality.
\\
The Lippmann-Schwinger equation (\ref{effective-equation_1}) in the Laplace-Fourier domain is expressed as follows
\[
    (\hbar\;\omega^2 + 1)\;\hat{\mathbf{Y}} + \omega^2\; \hat{\mathbcal{S}}_\omega\big(\hat{\mathbf{Y}}\big) = \omega^2\; \widehat{u}^\textit{in},\quad \text{on}\ \mathbf{\Gamma}.
\]
We proceed by showing that $\hat{\mathbf{Y}}$ satisfies the following variational formulation
\[
    \begin{cases}
        \varphi \in L^2(\mathbf{\Gamma})\ \text{such that}\\
        a_\omega(\hat{\mathbf{Y}},\varphi) := \big\langle (\hbar\;\omega^2 + 1)\hat{\mathbf{Y}}, \varphi \big\rangle + \omega^2 \big\langle \hat{\mathbcal{S}}_\omega\big(\hat{\mathbf{Y}}\big), \varphi \big\rangle = \omega^2 \big\langle \widehat{u}^\textit{in}, \varphi \big\rangle,\quad \forall \varphi \in L^2(\mathbf{\Gamma}).
    \end{cases}
\]
Next, to establish the coercivity of the sesquilinear form, we select $\varphi = i\omega\hat{\mathbf{Y}}$, yielding the following variational expression
\[
    a_\omega(\hat{\mathbf{Y}}, i\omega\hat{\mathbf{Y}}) := \big\langle (\hbar\;\omega^2 + 1)\hat{\mathbf{Y}}, i\omega\hat{\mathbf{Y}} \big\rangle + \omega^2 \big\langle \hat{\mathbcal{S}}_\omega\big(\hat{\mathbf{Y}}\big), i\omega\hat{\mathbf{Y}} \big\rangle = \omega^2 \big\langle \widehat{u}^\textit{in}, i\omega\hat{\mathbf{Y}} \big\rangle,\quad \forall \hat{\mathbf{Y}} \in L^2(\mathbf{\Gamma}),
\]
which simplifies to
\[
    a_\omega(\hat{\mathbf{Y}}, i\omega\hat{\mathbf{Y}}) := -i\omega |\omega|^2 \hbar \big\langle \hat{\mathbf{Y}}, \hat{\mathbf{Y}} \big\rangle - i \overline{\omega} \big\langle \hat{\mathbf{Y}}, \hat{\mathbf{Y}} \big\rangle + |\omega|^2 \Big(-i\omega \big\langle \hat{\mathbcal{S}}_\omega\big(\hat{\mathbf{Y}}\big), \hat{\mathbf{Y}} \big\rangle \Big) = -i \omega |\omega|^2 \big\langle \widehat{u}^\textit{in}, \hat{\mathbf{Y}} \big\rangle,\quad \forall \hat{\mathbf{Y}} \in L^2(\mathbf{\Gamma}).
\]
Referring to the results in \cite[Section 1.3]{coercive-1} and \cite[pp. 314]{Ha-Duong-1}, we infer the following
\[
    \Re\Big(-i\omega \big\langle \hat{\mathbcal{S}}_\omega\big(\hat{\mathbf{Y}}\big), \hat{\mathbf{Y}} \big\rangle \Big) \geq 0,\quad \forall \hat{\mathbf{Y}} \in L^2(\mathbf{\Gamma}).
\]
Thus, by taking the real part, we obtain
\begin{align}\label{esti1}
    \Re\Big( a_\omega(\hat{\mathbf{Y}}, i\omega\hat{\mathbf{Y}}) \Big) \geq \omega_0 |\omega|^2 \Vert \hat{\mathbf{Y}} \Vert^2_{L^2(\mathbf{\Gamma})}.
\end{align}
We also establish the following estimate
\begin{align}\label{esti2}
    \Big| -i\ \omega |\omega|^2 \big\langle \widehat{u}^\textit{in}, \hat{\mathbf{Y}} \big\rangle \Big| \leq |\omega|^3 \Vert \widehat{u}^\textit{in} \Vert_{L^2(\mathbf{\Gamma})} \Vert \hat{\mathbf{Y}} \Vert_{L^2(\mathbf{\Gamma})}.
\end{align}
Consequently, combining estimates (\ref{esti1}) and (\ref{esti2}), we deduce
\[
    \Vert \hat{\mathbf{Y}} \Vert_{L^2(\mathbf{\Gamma})} \leq \frac{|\omega|}{\sigma} \Vert \widehat{u}^\textit{in} \Vert_{L^2(\mathbf{\Gamma})},
\]
which proves the unique solvability of the Lippmann-Schwinger equation (\ref{effective-equation_1}) in the Fourier-Laplace domain. Furthermore, we obtain the following norm estimate
\begin{align}\label{e1}
    \Big\Vert \Big((\hbar\;\omega^2 + 1)\mathbf{I} + \omega^2\; \widehat{\bm{\mathcal{S}}}_\omega \Big)^{-1} \Big\Vert_{\mathrm{L}^2(\mathbf{\Gamma}) \to \mathrm{L}^2(\mathbf{\Gamma})} \leq \frac{|\omega|}{\sigma}.
\end{align}
To present the main result, we utilize the Lubich notation \cite{lubich}. Consequently, we state the following Lemma.
\begin{lemma}\label{lubich}\cite[Lemma 2.1]{lubich}
If \( F(\bm{s}): \mathbb{C}_+ := \{\bm{s} \in \mathbb{C} : \Re \bm{s} > 0\} \to X \), where \( X \) is a Hilbert space, as an analytic function that is assumed to be polynomially bounded. Specifically, for \( j \in \mathbb{R} \), and for all \( \sigma \), there exists a constant \( C_F(\sigma) < \infty \) such that  
\begin{equation}\label{1.1}
    \Vert F(\bm{s}) \Vert_X \leq C_F(\Re \bm s) |\bm{s}|^j, \quad \text{for} \ \Re\bm{s} \geq \sigma > 0,
\end{equation} 
then the operator \( F(\partial_t) \) extends by density to a bounded linear operator
\[
    F(\partial_t) : H_0^{r+j}(0,T; X) \to H^r_0(0,T; Y),\; \text{for any}\; r \in \mathbb{R}.
\]
\end{lemma}
\noindent
Subsequently, we define the operator $\hat{\mathbcal{A}}_\omega: \mathrm{L}^2(\mathbf{\Gamma}) \to \mathrm{L}^2(\mathbf{\Gamma})$ as follows
\[
    \hat{\mathbcal{A}}_\omega(\omega) := \Big( (\hbar\;\omega^2 + 1)\mathbf{I} + \omega^2\; \widehat{\bm{\mathcal{S}}}_\omega \Big)^{-1}.
\]
By applying the inverse Fourier-Laplace transform, we define the corresponding time-domain solution operator for the Lippmann-Schwinger equation (\ref{effective-equation_1}) as $\hat{\mathbcal{A}}_\omega(\partial_t)$. Then, based on the estimate (\ref{e1}) and Lemma \ref{lubich}, we conclude
\[
    \hat{\mathbcal{A}}_\omega(\partial_t) : \mathrm{H}^{\mathrm{r}+1}_{0,\sigma}(0,\mathrm{T}; \mathrm{L}^2(\mathbf{\Gamma})) \to \mathrm{H}^{\mathrm{r}}_{0,\sigma}(0,\mathrm{T}; \mathrm{L}^2(\mathbf{\Gamma}))
\]
is bounded.
\\
This completes the proof.

\end{proof}      

\noindent
We now proceed by stating the following Corollary, which builds upon Proposition (\ref{prop1}). This result is necessary for the next section, where higher regularity of the solution to the Lippmann-Schwinger equation (\ref{effective-equation_1}) is required.

\begin{corollary}\label{prop2}   

Consider the Lippmann-Schwinger equation (\ref{effective-equation_1}). The solution $\mathbf{Y}$ belongs to the space $\mathrm{H}^{4}_{0,\sigma}(0,\mathrm{T}; \mathrm{W}^{1,p}(\mathbf{\Gamma}))$ for $p \in (1,\infty)$. In addition to that, we have $\partial_t^3\mathbf{Y} \in \mathbcal C(0,\mathrm{T}; \mathrm{L}^\infty(\mathbf{\Gamma}))$.

\end{corollary}     

\begin{proof}        

We begin by reformulating the Lippmann-Schwinger equation (\ref{effective-equation_1}) as a second-order non-homogeneous ordinary differential equation
\begin{align}\label{ode}
    \begin{cases}
    \displaystyle
    \hbar\frac{\mathrm{d}^2}{\mathrm{d}\mathrm{t}^2}\mathbf{Y}(\cdot,\mathrm{t}) + \mathbf{Y}(\cdot,\mathrm{t}) = \mathbcal{F}(\cdot,\mathrm{t}) \quad \text{in} \ (0, \mathrm{T}), \\
    \mathbf{Y}(\mathrm{0}) = \frac{\mathrm{d}}{\mathrm{d}\mathrm{t}}\mathbf{Y}(\mathrm{0}) = 0,
    \end{cases}
\end{align}
where the term $\mathbcal{F}$ is defined as
\begin{equation}
    \mathbcal{F}(x,t) := \frac{\partial^2}{\partial t^2}u^\textbf{in}(\mathrm{x},\mathrm{t}) - \int_{\mathbf{\Gamma}}\frac{\overline{\bm{C}}}{4\pi\vert x-y\vert} \frac{\partial^2}{\partial t^2}{\mathrm{Y} (y, t-c_0^{-1}\vert x-y\vert)} \, dy.
\end{equation}
The incident wave field \(u^\textit{in}\) originates from a point source located at \(\mathrm{x}_0 \in \mathbb{R}^{3} \setminus \overline{\mathrm{D}}\) and is given by:
\[
u^\textit{in}(\mathrm{x}, \mathrm{t}, \mathrm{x}_0) := \rho_c\frac{\lambda(\mathrm{t} - \mathrm{c}_0^{-1} \vert \mathrm{x} - \mathrm{x}_0 \vert)}{\vert \mathrm{x} - \mathrm{x}_0 \vert},
\]
where \(\lambda \in \mathcal{C}^4(\mathbb{R})\) is a causal signal (i.e., it vanishes for \(\mathrm{t} < 0\)), and we assume \(u^\textit{in}\) belongs to the function space $\mathrm{H}^{5}_{0,\sigma}(0,\mathrm{T}; \mathrm{L}^2(\mathbf{\Gamma}))$, implying \(r=5\). Consequently, we have \(\partial_t^2u^\textit{in} \in \mathrm{H}^{3}_{0,\sigma}(0,\mathrm{T}; \mathrm{L}^2(\mathbf{\Gamma}))\).
\\
By Proposition \ref{prop1}, it follows that \(\mathbf{\mathrm{Y}} \in \mathrm{H}^{4}_{0,\sigma}(0,\mathrm{T}; \mathrm{L}^2(\mathbf{\Gamma}))\), which implies that \(\partial_t^2\mathbf{\mathrm{Y}} \in \mathrm{H}^{2}_{0,\sigma}(0,\mathrm{T}; \mathrm{L}^2(\mathbf{\Gamma}))\). Due to the smoothing properties of the single-layer operator, for \(\partial_t^2\mathbf{\mathrm{Y}} \in \mathrm{H}^{2}_{0,\sigma}(0,\mathrm{T}; \mathrm{L}^2(\mathbf{\Gamma}))\), we have \(\mathbcal{S}\big(\partial_t^2\mathbf{\mathrm{Y}}\big) \in \mathrm{H}^{2}_{0,\sigma}(0,\mathrm{T}; \mathrm{H}^1(\mathbf{\Gamma}))\), see for instance \cite{Q-S}.
\\
Given the smoothness of the incident wave field in the spatial domain, we further assume \(\partial_t^2u^\textit{in} \in \mathrm{H}^{2}_{0,\sigma}(0,\mathrm{T}; \mathrm{H}^1(\mathbf{\Gamma}))\), as $x_0$ is away from $\bf{\Gamma}$. Combining these observations, we deduce that \(\mathbcal{F} \in \mathrm{H}^{2}_{0,\sigma}(0,\mathrm{T}; \mathrm{H}^1(\mathbf{\Gamma}))\).
\\
Next, by applying the compact Sobolev embedding \(H^1(\mathbf{\Gamma}) \hookrightarrow L^p(\mathbf{\Gamma})\) for \(p \in (1, \infty)\), we further conclude that \(\mathbcal{F} \in \mathrm{H}^{2}_{0,\sigma}(0,\mathrm{T}; \mathrm{L}^p(\mathbf{\Gamma}))\) for \(p \in (1, \infty)\).
\\
Finally, considering that the single-layer operator in the Fourier-Laplace domain maps \(L^p(\mathbf{\Gamma})\) to \(\mathrm{W}^{1,p}(\mathbf{\Gamma})\) for \(p \in (1, \infty)\) (see \cite{mitrea}), and assuming additional smoothness of the incident wave, we conclude that
\begin{equation}\label{regularity-1}
    \mathbcal{F} \in \mathrm{H}^{2}_{0,\sigma}(0,\mathrm{T}; \mathrm{W}^{1,p}(\mathbf{\Gamma}))\; \text{for} \;p \in (1, \infty).
\end{equation}
Hence, due to the well-posedness of the non-homogeneous problem (\ref{ode}), we derive that
\begin{equation}
    \partial_t^2\mathbf{\mathrm{Y}} \in \mathrm{H}^{2}_{0,\sigma}(0,\mathrm{T}; \mathrm{W}^{1,p}(\mathbf{\Gamma}))\; \text{for}\ p \in (1, \infty),\; \text{i.e.}\; \mathbf{\mathrm{Y}} \in \mathrm{H}^{4}_{0,\sigma}(0,\mathrm{T}; \mathrm{W}^{1,p}(\mathbf{\Gamma}))\; \text{for}\ p \in (1, \infty),
\end{equation}
Which completes the first part of the proof.
\\
Applying the Sobolev embedding theorem \( H^4(0,T) \xhookrightarrow{} \mathbcal{C}^3(0,T) \), we deduce that \( \partial_t^3\mathbf{Y} \in \mathbcal{C}\big(0,T; \mathrm{W}^{1,p}(\mathbf{\Gamma})\big) \) for \( p \in (1, \infty) \). 
\\
Next, from the regularity condition in (\ref{regularity-1}), it follows that \( \mathbcal{F} \in H^{2}_{0,\sigma}(0,T; \mathrm{W}^{1,p}(\mathbf{\Gamma})) \) for \( p \in (1, \infty) \). Moreover, for \( p > 2 \), the continuous embedding \( \mathrm{W}^{1,p}(\mathbf{\Gamma}) \hookrightarrow L^\infty(\mathbf{\Gamma}) \) holds, recalling that \( \Gamma \) is \( \mathbcal{C}^2 \)-regular. This completes the proof. 
\end{proof}

\begin{corollary}\label{prop3}
    Consider the Lippmann-Schwinger equation (\ref{ef-eq-1}). Then, we have $\partial_x \mathbcal{F}_m \in \mathbcal{C}\big(0,\mathrm{T};\mathrm{L}^\infty(\mathbf{\Gamma})).$ 
\end{corollary}
\begin{proof}
    Similar to Corollary \ref{prop2}, we begin by reformulating the Lippmann-Schwinger equation (\ref{ef-eq-1}) as a second-order non-homogeneous ordinary differential equation
\begin{align}\label{ode-1}
    \begin{cases}
    \displaystyle
    \hbar\frac{\mathrm{d}^2}{\mathrm{d}\mathrm{t}^2}\mathbcal{F}_m(\cdot,\mathrm{t}) + \mathbcal{F}_m(\cdot,\mathrm{t}) = \mathbb{F}(\cdot,\mathrm{t}) \quad \text{in} \ (0, \mathrm{T}), \\
    \mathbcal{F}_m(\mathrm{0}) = \frac{\mathrm{d}}{\mathrm{d}\mathrm{t}}\mathbcal{F}_m(\mathrm{0}) = 0,
    \end{cases}
\end{align}
where the term $\mathbb{F}$ is defined as
\begin{equation}
    \mathbb{F}(x,t) := \frac{\partial^2}{\partial t^2}u^\textbf{in}(\mathrm{x},\mathrm{t}) - \int_{\mathbf{\Gamma}} 
         \begin{pmatrix}
               \frac{\overline{\bm{C}}}{4\pi|x-y|} \cdot \sum\limits_{l=1}^{\lfloor K(\cdot) + 1 \rfloor} \mathcal{F}_{m_l}(\mathrm{y}, \mathrm{t} - c_0^{-1} \vert x - y \vert) \\
               \vdots \\
               \frac{\overline{\bm{C}}}{4\pi|x-y|} \cdot \sum\limits_{l=1}^{\lfloor K(\cdot) + 1 \rfloor} \mathcal{F}_{m_l}(\mathrm{y}, \mathrm{t} - c_0^{-1} \vert x - y \vert)
        \end{pmatrix}.
\end{equation}
As previously discussed, \(\mathbcal{F}_m(\cdot, \cdot)\) can be expressed as the sum of its individual components, specifically \(\sum\limits_{\ell=1}^{\lfloor K(\cdot) + 1 \rfloor} \mathcal{F}_{m_\ell}(\cdot, \cdot)\). By recalling the definition of the averaged sum \(\bm{\mathrm{Y}} := \frac{1}{\lfloor K(\cdot) + 1 \rfloor} \sum\limits_{\ell=1}^{\lfloor K(\cdot) + 1 \rfloor} \mathcal{F}_{m_\ell}(x, t)\) and noting that \(\lfloor K(\cdot) + 1 \rfloor \in L^\infty(\bm{\Gamma})\), we observe that the regularity properties established for \(\bm{\mathrm{Y}}\) implies \(\sum\limits_{l=1}^{\lfloor K(\cdot) + 1 \rfloor} \partial_t \mathcal{F}_{m_l} \in \mathbcal C(0,\mathrm{T}; \mathrm{L}^\infty(\mathbf{\Gamma})) \) as well as $\mathbcal{F}_m \in \mathbcal C(0,\mathrm{T}; \mathrm{L}^\infty(\mathbf{\Gamma}))$. Now, upon taking the partial derivative with respect to \( x \), we obtain the following estimate:
\begin{align}
    &\nonumber\big|\partial_{x_i} \mathbb{F}(x,t)\big| 
    \\ &\nonumber\lesssim \int_{\mathbf{\Gamma}} \begin{pmatrix}
               \frac{\overline{\bm{C}}}{4\pi|x-y|^2} \cdot \Big|\sum\limits_{l=1}^{\lfloor K(\cdot) + 1 \rfloor} \mathcal{F}_{m_l}(\mathrm{y}, \mathrm{t} - c_0^{-1} \vert x - y \vert)\Big| \\
               \vdots \\
               \frac{\overline{\bm{C}}}{4\pi|x-y|^2} \cdot \Big|\sum\limits_{l=1}^{\lfloor K(\cdot) + 1 \rfloor} \mathcal{F}_{m_l}(\mathrm{y}, \mathrm{t} - c_0^{-1} \vert x - y \vert)\Big|
        \end{pmatrix}
        + \int_{\mathbf{\Gamma}} \begin{pmatrix}
               \frac{\overline{\bm{C}}}{4\pi|x-y|} \cdot \Big|\sum\limits_{l=1}^{\lfloor K(\cdot) + 1 \rfloor} \partial_t\mathcal{F}_{m_l}(\mathrm{y}, \mathrm{t} - c_0^{-1} \vert x - y \vert)\Big| \\
               \vdots \\
               \frac{\overline{\bm{C}}}{4\pi|x-y|} \cdot \Big|\sum\limits_{l=1}^{\lfloor K(\cdot) + 1 \rfloor} \partial_t\mathcal{F}_{m_l}(\mathrm{y}, \mathrm{t} - c_0^{-1} \vert x - y \vert)\Big|
        \end{pmatrix}
        + \partial_{x_i} \partial_t^2u^\textit{in} 
    \\ &\nonumber \lesssim \int_{\mathbf{\Gamma}} \begin{pmatrix}
               \frac{\overline{\bm{C}}}{4\pi|x-y|^2} \cdot \Big\Vert\sum\limits_{l=1}^{\lfloor K(\cdot) + 1 \rfloor} \mspace{-20mu}\mathcal{F}_{m_l}(\mathrm{y}, \mathrm{t} - c_0^{-1} \vert x - y \vert)\Big\Vert_{C(0,T;L^\infty(\mathbf{\Gamma}))} \\
               \vdots \\
               \frac{\overline{\bm{C}}}{4\pi|x-y|^2} \cdot \Big\Vert\sum\limits_{l=1}^{\lfloor K(\cdot) + 1 \rfloor}\mspace{-20mu} \mathcal{F}_{m_l}(\mathrm{y}, \mathrm{t} - c_0^{-1} \vert x - y \vert)\Big\Vert_{C(0,T;L^\infty(\mathbf{\Gamma}))}
        \end{pmatrix}
        \\ &\nonumber+ \begin{pmatrix} \displaystyle
               \Big(\int_{\mathbf{\Gamma}} \frac{1}{|x-y|^2}\Big)^\frac{1}{2} \cdot \Big\Vert\sum\limits_{l=1}^{\lfloor K(\cdot) + 1 \rfloor}\mspace{-20mu} \partial_t\mathcal{F}_{m_l}(\mathrm{y}, \mathrm{t} - c_0^{-1} \vert x - y \vert)\Big\Vert_{C(0,T;L^\infty(\mathbf{\Gamma}))} \\
               \vdots \\
               \displaystyle\Big(\int_{\mathbf{\Gamma}} \frac{1}{|x-y|^2}\Big)^\frac{1}{2} \cdot \Big\Vert\sum\limits_{l=1}^{\lfloor K(\cdot) + 1 \rfloor} \mspace{-20mu}\partial_t\mathcal{F}_{m_l}(\mathrm{y}, \mathrm{t} - c_0^{-1} \vert x - y \vert)\Big\Vert_{C(0,T;L^\infty(\mathbf{\Gamma}))}
        \end{pmatrix}
        + \mathcal{O}(1) = \mathcal{O}(1),
\end{align}
which implies that \( \partial_{x_i}\mathbf{F} \in H^2_{0,\sigma}\big(0,T;L^\infty(\mathbf{\Gamma})\big) \). By repeating these steps and following similar procedures in the context of the well-posedness of the non-homogeneous differential equation (\ref{ode}), we conclude that \( \partial_{x_i}\mathbcal{F}_m \in \mathbcal{C}\big(0,T;L^\infty(\mathbf{\Gamma})\big) \). This completes the proof.
\end{proof}     


    \subsection{The asymptotic approximations: Proof of Theorem \ref{mainth}} 

    \subsubsection{The generated effective medium} \label{gen-eff-med}   

\noindent
To begin, we recall Proposition \ref{main-prop}, where \(M = [d^{-2}]\), providing an asymptotic expansion for the wave field, given by
\begin{align} \label{asymptotic-expansion-us}
    u^\textit{sc}(x, t) = -\sum_{m=1}^{[d^{-2}]} \frac{\mathbf{C}_m}{4\pi |x - \mathbf{z}_m|} \mathrm{Y}_m\big(t - c_0^{-1} |x - \mathbf{z}_m|\big) + \mathcal{O}(M\varepsilon^{2}), \quad \text{as} \quad \varepsilon\ll 1.
\end{align}
Due to the spatial distribution of the bubbles as specified in Assumption \textcolor{blue}{1}, we reformulate this expression as follows:
\begin{align}\label{ef-es}
    u^\textit{sc}(x, t) 
    \nonumber&= -\sum_{j=1}^{[d^{-2}]} \frac{\mathbf{C}_{m_l}}{4\pi |x - \mathbf{z}_{m_1}|} \sum_{l=1}^{\lfloor K(\cdot) + 1\rfloor}\mathrm{Y}_{m_l}\big(t - c_0^{-1} |x - \mathbf{z}_{m_l}|\big)
    \\ \nonumber&- \sum_{m=1}^{[d^{-2}]} \sum_{l=1}^{\lfloor K(\cdot) + 1\rfloor}\frac{\mathbf{C}_{m_l}}{4\pi} \Big(\frac{1}{|x - \mathbf{z}_{m_l}|}-\frac{1}{|x - \mathbf{z}_{m_1}|}\Big)\mathrm{Y}_{m_l}\big(t - c_0^{-1} |x - \mathbf{z}_{m_l}|\big) + \mathcal{O}(M\varepsilon^{2})
    \\ \nonumber&= -\sum_{m=1}^{[d^{-2}]} \frac{\mathbf{C}_{m_l}}{4\pi |x - \mathbf{z}_{m_1}|} \sum_{l=1}^{\lfloor K(\cdot) + 1\rfloor}\mathcal{F}_{m_l}\big(z_{m_l},t - c_0^{-1} |x - \mathbf{z}_{m_l}|\big)
    \\ \nonumber&- \underbrace{\sum_{m=1}^{[d^{-2}]} \frac{\mathbf{C}_{m_l}}{4\pi |x - \mathbf{z}_{m_1}|}\sum_{l=1}^{\lfloor K(\cdot) + 1\rfloor}\Big(\mathrm{Y}_{m_l}\big(t - c_0^{-1} |x - z_{m_l}|\big)-\mathcal{F}_{m_l}\big(z_{m_l},t - c_0^{-1} |x - \mathbf{z}_{m_l}|\big)\Big)}_{:= \textit{Error}_{(1)}}
    \\ \nonumber&- \underbrace{\sum_{m=1}^{[d^{-2}]} \sum_{l=1}^{\lfloor K(\cdot) + 1\rfloor}\frac{\mathbf{C}_{m_l}}{4\pi} \Big(\frac{1}{|x - \mathbf{z}_{m_l}|}-\frac{1}{|x - \mathbf{z}_{m_1}|}\Big)\mathcal{F}_{m_l}\big(z_{m_l},t - c_0^{-1} |x - \mathbf{z}_{m_l}|\big)}_{:= \textit{Error}_{(2)}}
    \\ \nonumber&- \underbrace{\sum_{m=1}^{[d^{-2}]} \sum_{l=1}^{\lfloor K(\cdot) + 1\rfloor}\frac{\mathbf{C}_{m_l}}{4\pi} \Big(\frac{1}{|x - \mathbf{z}_{m_l}|}-\frac{1}{|x - \mathbf{z}_{m_1}|}\Big)\Bigg(\mathrm{Y}_{m_l}(t - c_0^{-1} |x - \mathbf{z}_{m_l}|)-\mathcal{F}_{m_l}\big(z_{m_l},t - c_0^{-1} |x - \mathbf{z}_{m_l}|\big)\Bigg)}_{:= \textit{Error}_{(3)}} 
    \\ &+ \mathcal{O}(M\varepsilon^{2}).
\end{align}
In light of the previous discussion on deriving equation (\ref{ef-eq-2}) from equation (\ref{ef-eq-1}), and observing that \(\bm{\mathbcal{F}}_j(\cdot.\cdot)\) can be expressed as the sum of its components, \(\sum\limits_{l=1}^{\lfloor K(\cdot) + 1\rfloor} \mathcal{F}_{m_l}(\cdot,\cdot)\), we represent the unknown term \(\sum\limits_{l=1}^{\lfloor K(\cdot) + 1\rfloor} \mathcal{F}_{m_l}\) as \(\bm{\mathbcal{F}}_m(\cdot,\cdot)\). A similar argument applies to \(\sum\limits_{l=1}^{\lfloor K(\cdot) + 1\rfloor} \mathrm{Y}_{m_l}\), which we can represent as \(\bm{\mathbb{Y}}_j\). Additionally, recalling that \(\mathbf{C}_{m_l} := \overline{\mathbf{C}}_{m_l} d^2\), we proceed with the following estimation:
\begin{enumerate}
    \item Estimation of $\textit{Error}_{(1)}.$
        \begin{align}\label{ef-es-1}
              |\textit{Error}_{(1)}| 
              &\nonumber\lesssim \Big|\sum_{m=1}^{[d^{-2}]} \frac{\mathbf{C}_{m_l}}{4\pi |x - \mathbf{z}_{m_1}|}\sum_{l=1}^{\lfloor K(\cdot) + 1\rfloor}\Big(\mathrm{Y}_{m_l}\big(t - c_0^{-1} |x - \mathbf{z}_{m_l}|\big)-\mathcal{F}_{m_l}\big(z_{m_l},t - c_0^{-1} |x - \mathbf{z}_{m_l}|\big)\Big)\Big|
              \\ \nonumber&\lesssim d^2 \Bigg(\sum_{m=1}^{[d^{-2}]}\sum_{l=1}^{\lfloor K(\cdot) + 1\rfloor}\frac{1}{|x - \mathbf{z}_{m_1}|^2}\Bigg)^\frac{1}{2} \ \Bigg(\sum_{m=1}^{[d^{-2}]}\big|\bm{\mathbcal{F}}_m - \bm{\mathbb{Y}}_m\big|^2\Bigg)^\frac{1}{2}
              \\ &\lesssim d\ |\ln(d)|^\frac{1}{2} \Bigg(\sum_{m=1}^{[d^{-2}]}\big|\bm{\mathbcal{F}}_m - \bm{\mathbb{Y}}_m\big|^2\Bigg)^\frac{1}{2}.
        \end{align}
    \item Estimation of $\textit{Error}_{(2)}.$ 
        \begin{align*}
              |\textit{Error}_{(2)}|
              &\nonumber\lesssim \Big|\sum_{m=1}^{[d^{-2}]} \sum_{l=1}^{\lfloor K(\cdot) + 1\rfloor}\frac{\mathbf{C}_{m_l}}{4\pi} \Big(\frac{1}{|x - \mathbf{z}_{m_l}|}-\frac{1}{|x - \mathbf{z}_{m_1}|}\Big)\mathcal{F}_{m_l}\big(z_{m_l},t - c_0^{-1} |x - \mathbf{z}_{m_l}|\big)\Big|
              \\ &\nonumber\lesssim d^2\ \Bigg(\sum_{m=1}^{[d^{-2}]}\sum_{l=1}^{\lfloor K(\cdot) + 1\rfloor}|x - \mathbf{z}_{m_1}|^2\Bigg)^\frac{1}{2} \ \Bigg(\sum_{m=1}^{[d^{-2}]}\big|\bm{\mathbcal{F}}_m \big|^2\Bigg)^\frac{1}{2}
              \\ &\lesssim d^2 \Bigg(\sum_{m=1}^{[d^{-2}]}\big|\bm{\mathbcal{F}}_m \big|^2\Bigg)^\frac{1}{2}.
        \end{align*}
        Then, based on the regularity result, we have $\bm{\mathbcal{F}}_m\in L^\infty\big(0,T;L^\infty(\Gamma_m)\big),$ we arrive at
        \begin{align}\label{ef-es-2}
            |\textit{Error}_{(2)}|
              \lesssim d\ \big\Vert \bm{\mathbcal{F}}_m\big\Vert_{L^\infty\big(0,T;L^\infty(\Gamma_m)\big)}.
        \end{align}
    \item Estimation of $\textit{Error}_{(3)}.$ 
        \begin{align}\label{ef-es-3}
              |\textit{Error}_{(3)}|
              &\nonumber\lesssim \Bigg|\sum_{m=1}^{[d^{-2}]} \sum_{l=1}^{\lfloor K(\cdot) + 1\rfloor}\frac{\mathbf{C}_{m_l}}{4\pi} \Big(\frac{1}{|x - \mathbf{z}_{m_l}|}-\frac{1}{|x - \mathbf{z}_{m_1}|}\Big)\Bigg(\mathrm{Y}_{m_l}(t - c_0^{-1} |x - \mathbf{z}_{m_l}|)-\mathcal{F}_{m_l}\big(z_{m_l},t - c_0^{-1} |x - \mathbf{z}_{m_l}|\big)\Bigg)\Bigg|
              \\ &\nonumber\lesssim d^2\ \Bigg(\sum_{m=1}^{[d^{-2}]}\sum_{l=1}^{\lfloor K(\cdot) + 1\rfloor}|x - \mathbf{z}_{m_1}|^2\Bigg)^\frac{1}{2} \ \Bigg(\sum_{m=1}^{[d^{-2}]}\big|\bm{\mathbcal{F}}_m- \bm{\mathbb{Y}}_m \big|^2\Bigg)^\frac{1}{2}
              \\ &\lesssim d^2\ \Bigg(\sum_{m=1}^{[d^{-2}]}\big|\bm{\mathbcal{F}}_m- \bm{\mathbb{Y}}_m\big|^2\Bigg)^\frac{1}{2}.
        \end{align}
\end{enumerate}
Consequently, combining the estimates (\ref{ef-es-1}), (\ref{ef-es-2}), (\ref{ef-es-3}) and plugging them in (\ref{ef-es}), we deduce that
    \begin{align}\label{th-1}
          u^\textit{sc}(x,t) 
          \nonumber&= -\sum_{m=1}^{[d^{-2}]} \frac{\mathbf{C}_{m_l}}{4\pi |x - \mathbf{z}_{j_1}|} \sum_{l=1}^{\lfloor K(\cdot) + 1\rfloor}\mathcal{F}_{m_l}\big(z_{m_l},t - c_0^{-1} |x - \mathbf{z}_{m_l}|\big) 
          + \mathcal{O}\Bigg(d\ |\ln(d)|^\frac{1}{2} \Bigg(\sum_{j=1}^{[d^{-2}]}\big|\bm{\mathbcal{F}}_m - \bm{\mathbb{Y}}_m\big|^2\Bigg)^\frac{1}{2}\Bigg)
          \\ &+ \mathcal{O}\Bigg(d^2\ \Bigg(\sum_{m=1}^{[d^{-2}]}\big|\bm{\mathbcal{F}}_m- \bm{\mathbb{Y}}_m\big|^2\Bigg)^\frac{1}{2}\Bigg)  
          + \mathcal{O}\Big(d\ \big\Vert \bm{\mathbcal{F}}_m\big\Vert_{L^\infty\big(0,T;L^\infty(\Gamma_m)\big)}\Big) + \mathcal{O}(M\varepsilon^{2}).
    \end{align}
Then, knowing the fact that \(\mathbf{Y} = \frac{\partial^2}{\partial t^2} \mathbf{U}\), from the equations (\ref{effective-equation})-(\ref{effective-equation-2})-(\ref{effective-equation-3}), we arrive at the following expression
     \begin{align}
          \mathbcal{W}^\textit{sc}(x,t) = -\int_{\mathbf{\Gamma}} \frac{\overline{\bm{C}}\ (\lfloor K(y) + 1 \rfloor)}{4\pi|x-y|} \bm{\mathrm{Y}}(y, t-c_0^{-1} \vert x-y \vert) dy
     \end{align}
Now, recalling the fact that
$$\bm{\Gamma}=  \Big(\bigcup_{m= 1}^{[d^{-2}]} \Gamma_m\Big)\cup \Big(\bm{\Gamma}\setminus\bigcup_{m=1}^{[d^{-2}]}\Gamma_m\Big),$$ we deduce
    \begin{align}\label{th-2}
        \mathbcal{W}^\textit{sc}(x,t) 
        = -\sum_{m=1}^{[d^{-2}]} \int_{\Gamma_m} \frac{\overline{\bm{C}}\ (\lfloor K(y) + 1 \rfloor)}{4\pi|x-y|} \bm{\mathrm{Y}}(y, t-c_0^{-1} \vert x-y \vert) dy -\int_{\bm{\Gamma}\setminus\bigcup\limits_{m=1}^{[d^{-2}]}\Gamma_m} \frac{\overline{\bm{C}}\ (\lfloor K(y) + 1 \rfloor)}{4\pi|x-y|} \bm{\mathrm{Y}}(y, t-c_0^{-1} \vert x-y \vert) dy
    \end{align}
Based on the estimate derived for the expression \( \textit{\textbf{Err}}^{(2)} \) (Equation \(\ref{ll1}\)) and the fact that \( \lfloor K(y) + 1 \rfloor \in L^\infty(\bm{\Gamma}) \), a similar reasoning allows us to deduce the following estimate:
\begin{align}
    \int_{\bm{\Gamma}\setminus\bigcup\limits_{m=1}^{[d^{-2}]}\Gamma_j} \frac{\overline{\bm{C}}\ (\lfloor K(y) + 1 \rfloor)}{4\pi|x-y|} \bm{\mathrm{Y}}(y, t-c_0^{-1} \vert x-y \vert) dy\lesssim \big\Vert\bm{\mathrm{Y}} \Vert_{L^\infty\big(0,\mathrm{T};\mathrm{L}^\infty(\Gamma_m)\big)}\ d.
\end{align}
Then, based on the above estimate, subtracting (\ref{th-1}) from (\ref{th-2}), we arrive at
\begin{align}
    \nonumber&u^\textit{sc}(x,t) - \mathbcal{W}^\textit{sc}(x,t) 
    \\ \nonumber&= \sum_{m=1}^{[d^{-2}]} \int_{\Gamma_m} \frac{\overline{\bm{C}}\ (\lfloor K(y) + 1 \rfloor)}{4\pi|x-y|} \bm{\mathrm{Y}}(y, t-c_0^{-1} \vert x-y \vert) dy 
    -\sum_{m=1}^{[d^{-2}]} \frac{\mathbf{C}_{m_l}}{4\pi |x - \mathbf{z}_{m_1}|} \sum_{l=1}^{\lfloor K(\cdot) + 1\rfloor}\mathcal{F}_{m_l}\big(z_{m_l},t - c_0^{-1} |x - \mathbf{z}_{m_l}|\big) 
    \\ \nonumber &+ \mathcal{O}\Bigg(d\ |\ln(d)|^\frac{1}{2} \Bigg(\sum_{m=1}^{[d^{-2}]}\big|\bm{\mathbcal{F}}_m - \bm{\mathbb{Y}}_m\big|^2\Bigg)^\frac{1}{2}\Bigg)
    + \mathcal{O}\Big(d\ \big\Vert \bm{\mathbcal{F}}_m\big\Vert_{L^\infty\big(0,T;L^\infty(\Gamma_m)\big)}\Big) 
    + \mathcal{O}\Bigg(d\ \big\Vert\bm{\mathrm{Y}} \Vert_{L^\infty\big(0,\mathrm{T};\mathrm{L}^\infty(\Gamma_m)\big)}\Bigg)
    \\ &+ \mathcal{O}\Bigg(d^2\ \Bigg(\sum_{m=1}^{[d^{-2}]}\big|\bm{\mathbcal{F}}_m- \bm{\mathbb{Y}}_m\big|^2\Bigg)^\frac{1}{2}\Bigg)+ \mathcal{O}(M\varepsilon^{2}).
\end{align}
Furthermore, recalling the definition \( \bm{\mathrm{Y}}(\cdot,\cdot) = \frac{1}{\lfloor K(\cdot) + 1 \rfloor} \sum\limits_{l=1}^{\lfloor K(\cdot) + 1 \rfloor} \mathcal{F}_{m_l}(\cdot, \cdot)\), using the fact that $\mathbcal{F}_m(\cdot,\cdot) = \sum\limits_{l=1}^{\lfloor K(\cdot) + 1 \rfloor} \mathcal{F}_{m_l}(\cdot, \cdot)$, and $\textit{Area}(\Gamma_m) = d^2$, we derive that
\begin{align}\label{mor}
    \nonumber&u^\textit{sc}(x,t) - \mathbcal{W}^\textit{sc}(x,t) 
    \\ \nonumber&= \sum_{m=1}^{[d^{-2}]}\int_{\Gamma_m} \frac{\overline{\bm{C}}}{4\pi}\Bigg( \frac{1}{|x-y|} \bm{\mathbcal{F}}_m(y, t-c_0^{-1} \vert x-y \vert)
    -\frac{1}{ |x - \mathbf{z}_{m}|} \bm{\mathbcal{F}}_m\big(z_{m},t - c_0^{-1} |x - \mathbf{z}_{m}|\big)\Bigg) dy
    \\ \nonumber &+ \mathcal{O}\Bigg(d\ |\ln(d)|^\frac{1}{2} \Bigg(\sum_{m=1}^{[d^{-2}]}\big|\bm{\mathbcal{F}}_m - \bm{\mathbb{Y}}_m\big|^2\Bigg)^\frac{1}{2}\Bigg)
    + \mathcal{O}\Big(d\ \big\Vert \bm{\mathbcal{F}}_m\big\Vert_{L^\infty\big(0,T;L^\infty(\Gamma_j)\big)}\Big) + \mathcal{O}\Bigg(d\ \big\Vert\bm{\mathrm{Y}} \Vert_{L^\infty\big(0,\mathrm{T};\mathrm{L}^\infty(\Gamma_j))}\Bigg)
    \\ &+ \mathcal{O}\Bigg(d^2\ \Bigg(\sum_{m=1}^{[d^{-2}]}\big|\bm{\mathbcal{F}}_m- \bm{\mathbb{Y}}_m\big|^2\Bigg)^\frac{1}{2}\Bigg)+ \mathcal{O}\Big(d\ |\log(d)|\ \big\Vert \bm{\mathbcal{F}}_m \Vert_{L^\infty\big(0,\mathrm{T};\mathrm{L}^\infty(\Gamma_m)\big)}\Big) + \mathcal{O}(M\varepsilon^{2}).
\end{align}
In order to estimate the first term of the right hand side of the above expression, we rewrite the integrand as follows:
    \begin{align}
          \nonumber&\frac{1}{|x-y|} \bm{\mathbcal{F}}_m(y, t-c_0^{-1} \vert x-y \vert)
    -\frac{1}{ |x - \mathbf{z}_{m}|} \bm{\mathbcal{F}}_m\big(z_m,t - c_0^{-1} |x - \mathbf{z}_{m}|\big)
          \\ \nonumber &= \bm{\mathbcal{F}}_m\big(z_m,t - c_0^{-1} |x - \mathbf{z}_m|\big) \Bigg(\frac{1}{|x-y|}-\frac{1}{|x - \mathbf{z}_m|}\Bigg) 
         + \frac{1}{|x-y|}\Big(\bm{\mathbcal{F}}_m(y, t-c_0^{-1} \vert x-y \vert)-\bm{\mathbcal{F}}_m\big(z_m,t - c_0^{-1} |x - \mathbf{z}_m|\big)\Big).
    \end{align}
Consequently, we do the following estimates:
\begin{enumerate}
    \item Estimation of $\textit{Error}^{(6)}.$ We have
        \begin{align}\label{error6}
            |\textit{Error}^{(6)}| 
            \nonumber&\lesssim \Bigg|\sum_{m=1}^{[d^{-2}]}\int_{\Gamma_m} \frac{\overline{\bm{C}}}{4\pi} \Bigg(\frac{1}{|x-y|}-\frac{1}{|x - \mathbf{z}_{m}|}\Bigg)\bm{\mathbcal{F}}_m\big(z_m,t - c_0^{-1} |x - \mathbf{z}_m|\big) \Bigg|
            \\ &\lesssim \Big|\sum_{m=1}^{[d^{-2}]} \frac{1}{d_{mj}^{2}} \int_{\Gamma_m} |y- z_{m}|\ dy\Big| \lesssim d\ |\ln(d)|\ \big\Vert \bm{\mathbcal{F}}_m \Vert_{L^\infty\big(0,\mathrm{T};\mathrm{L}^\infty(\Gamma_m)\big)}.
        \end{align}
    \item  Estimation of $\textit{Error}^{(7)}.$ First, We note that
        \begin{align*}
            \nonumber&\Big(\bm{\mathbcal{F}}_m(y, t-c_0^{-1} \vert x-y \vert)- \bm{\mathbcal{F}}_m\big(z_m,t - c_0^{-1} |x - \mathbf{z}_m|\big)\Big) 
            \\ &= (y-z_m)\;\partial_x\mathbcal{F}_m (z^*, t-c_0^{-1}|x-z_m|) + (y-z_m)\; \nabla_y|y-z^*|\;\partial_t\mathbcal{F}_m (y, t^*),
        \end{align*}
        where, $z^*\in \Omega_m$ and $t^*\in (t-c_0^{-1}|x-z_m|,t-c_0^{-1}|x-y|).$ Consequently, we derive that
        \begin{align}\label{error7}
             |\textit{Error}^{(7)}| 
            \nonumber&\lesssim \sum_{m=1}^{[d^{-2}]}\int_{\Gamma_m} \frac{\overline{\bm{C}}}{4\pi} \frac{1}{|x-y|}\Big(\bm{\mathbcal{F}}_m(y, t-c_0^{-1} \vert x-y \vert)- \bm{\mathbcal{F}}_m\big(z_m,t - c_0^{-1} |x - \mathbf{z}_m|\big)\Big)
            \\ \nonumber &\lesssim d\big\Vert \partial_x\bm{\mathbcal{F}}_m \Vert_{L^\infty\big(0,\mathrm{T};\mathrm{L}^\infty(\Gamma_m)\big)} 
            + d^3\ \sum_{m=1}^{[d^{-2}]} \frac{1}{d_{mj}} \big\Vert \partial_t\bm{\mathbcal{F}}_m \Vert_{L^\infty\big(0,\mathrm{T};\mathrm{L}^\infty(\Gamma_m)\big)}
            \\ &\lesssim d\big\Vert \partial_x\bm{\mathbcal{F}}_m\Vert_{L^\infty\big(0,\mathrm{T};\mathrm{L}^\infty(\Gamma_m)\big)} + d\ \big\Vert \partial_t\bm{\mathbcal{F}}_m \Vert_{L^\infty\big(0,\mathrm{T};\mathrm{L}^\infty(\Gamma_m)\big)}.
        \end{align}
\end{enumerate}
Then, having the estimates (\ref{error6}) and (\ref{error7}), we obtain
\begin{align}
    \nonumber&\sum_{m=1}^{[d^{-2}]}\int_{\Gamma_m} \frac{\overline{\bm{C}}}{4\pi}\Bigg(\frac{1}{|x-y|} \bm{\mathbcal{F}}_m(y, t-c_0^{-1} \vert x-y \vert)- \frac{1}{ |x - \mathbf{z}_{m_1}|} \bm{\mathbcal{F}}_m\big(z_m,t - c_0^{-1} |x - \mathbf{z}_m|\big) 
    \Bigg) dy 
    \\ &\lesssim d\ |\ln(d)|\ \big\Vert \bm{\mathbcal{F}}_m \Vert_{L^\infty\big(0,\mathrm{T};\mathrm{L}^\infty(\Gamma_m)\big)} 
    + d \big\Vert \partial_x\bm{\mathbcal{F}}_m \Vert_{L^\infty\big(0,\mathrm{T};\mathrm{L}^\infty(\Gamma_m)\big)} + d\ \big\Vert \partial_t\bm{\mathbcal{F}}_m\Vert_{L^\infty\big(0,\mathrm{T};\mathrm{L}^\infty(\Gamma_m)\big)}.
\end{align}
Based on the above estimate, we then finally obtain that
\begin{align}
    \nonumber&u^\textit{sc}(x,t) - \mathbcal{W}^\textit{sc}(x,t) 
    \\ \nonumber&= \mathcal{O}\Big(d\ |\ln(d)|\ \big\Vert \bm{\mathbcal{F}}_m \Vert_{L^\infty\big(0,\mathrm{T};\mathrm{L}^\infty(\Gamma_m)\big)}\Big) 
    + \Big(d\ \big\Vert \partial_x\bm{\mathbcal{F}}_m \Vert_{L^\infty\big(0,\mathrm{T};\mathrm{L}^\infty(\Gamma_m)\big)}\Big) 
    + \Big(d\ \big\Vert \partial_t\bm{\mathbcal{F}}_m \Vert_{L^\infty\big(0,\mathrm{T};\mathrm{L}^\infty(\Gamma_m)\big)}\Big)
    \\ \nonumber &+ \mathcal{O}\Bigg(d\ |\ln(d)|^\frac{1}{2} \Bigg(\sum_{m=1}^{[d^{-2}]}\big|\bm{\mathbcal{F}}_m - \bm{\mathbb{Y}}_m\big|^2\Bigg)^\frac{1}{2}\Bigg)
    + \mathcal{O}\Big(d\ \big\Vert \bm{\mathbcal{F}}_m\big\Vert_{L^\infty\big(0,T;L^\infty(\Gamma_m)\big)}\Big) 
    + \mathcal{O}\Bigg(d\ \big\Vert\bm{\mathrm{Y}} \Vert_{L^\infty\big(0,\mathrm{T};\mathrm{L}^\infty(\Gamma_m)\big)}\Bigg)
    \\ &+ \mathcal{O}\Bigg(d^2\ \Bigg(\sum_{m=1}^{[d^{-2}]}\big|\bm{\mathbcal{F}}_m- \bm{\mathbb{Y}}_m\big|^2\Bigg)^\frac{1}{2}\Bigg)
    + \mathcal{O}\Big(d\ |\log(d)|\ \big\Vert \bm{\mathbcal{F}}_m \Vert_{L^\infty\big(0,\mathrm{T};\mathrm{L}^\infty(\Gamma_m)\big)}\Big) + \mathcal{O}(M\varepsilon^{2}).
\end{align}
In order to complete the proof of the Theorem, we state the following proposition.

\begin{proposition}\label{mor-prop}    

We consider the algebric system and the corresponding general Lippmann-Schwinger equation described in (\ref{tran1}) and (\ref{surface-integral-equn}), respectively as follows:
\begin{align}\label{tran3}
    \begin{cases}
             \mathbb{A}\cdot\frac{\mathrm{d}^2}{\mathrm{d}\mathrm{t}^2}\bm{\mathbb{Y}}_m(\mathrm{t}) + \bm{\mathbb{Y}}_m(\mathrm{t}) + \sum\limits_{\substack{j=1 \\ j\neq m}}^{[d^{-2}]}\mathbb C_{mj} \cdot\frac{\mathrm{d}^2}{\mathrm{d}\mathrm{t}^2}\mathbb Y_j(\mathrm{t}-\mathrm{c}_0^{-1}|\mathrm{z}_{m}-\mathrm{z}_{j}|) = \frac{\mathrm{d}^2}{\mathrm{d}\mathrm{t}^2} \mathbcal H_m^\textit{in}\; \mbox{ in } (0, \mathrm{T}),
             \\ \bm{\mathbb{Y}}_m(\mathrm{0}) = \frac{\mathrm{d}}{\mathrm{d}\mathrm{t}}\bm{\mathbb{Y}}_m(\mathrm{0}) = 0, 
    \end{cases}
\end{align}
and for $(\mathrm{x},t) \in \mathbb{R}^3\times (0, T)$
\begin{align}
    \rchi_{\mathbf{\Gamma}} \, \mathbb{A} \cdot \frac{\partial^2}{\partial t^2} \mathbcal{F}_m (\mathrm{x}, \mathrm{t}) + \mathbcal{F}_m (\mathrm{x}, \mathrm{t}) + \int_{\mathbf{\Gamma}} \mathbcal{C}(x, y) \cdot \frac{\partial^2}{\partial t^2} \mathbcal{F}_m(x, t - c_0^{-1} \vert x - y \vert) \, dy = \frac{\partial^2}{\partial t^2} \mathbcal{F}_m^\textbf{in}(\mathrm{x}, \mathrm{t}).
\end{align}
Then, the following estimate holds
 \begin{align}
     \sum_{m=1}^{[d^{-2}]}\big|\bm{\mathbcal{F}}_m(z_m,t) - \bm{\mathbb{Y}}_m(t)\big|^2 \lesssim  d^2\sum_{m=1}^{[d^{-2}]}\big\Vert \frac{\partial^2}{\partial t^2} \mathbcal{F}_m \Vert_{L^\infty\big(0,\mathrm{T};\mathrm{L}^\infty(\Gamma_m)\big)}^2.
 \end{align}
 
\end{proposition}         

\begin{proof}
    See Section \ref{mor-sub} for more details.
\end{proof}

\noindent
Then, based on the Proposition \ref{mor-prop}, Lemma \ref{prop1}, Corollary \ref{prop2} and considering the regime
$$ d \sim \varepsilon^\frac{1}{2},$$
we finally obtain that
\begin{align}
    \nonumber&u^\textit{sc}(x,t) - \mathbcal{W}^\textit{sc}(x,t) = \mathcal{O}(\varepsilon^\frac{1}{2}).
\end{align}
This completes the proof of Theorem \ref{mainth}. \qed


    \subsubsection{Proof of the Proposition \ref{mor-prop}}\label{mor-sub}  

\noindent
The following step involves discretizing the surface integral equation given by (\ref{surface-integral-equn}), and estimating the difference between this surface integral and the general algebraic system given by (\ref{tran1}). Let us first rewrite the expression (\ref{surface-integral-equn}), using the fact that
$$\bm{\Gamma}= \Gamma_m\cup \Big(\bigcup_{j\ne m}^{[d^{-2}]} \Gamma_j\Big)\cup \Big(\bm{\Gamma}\setminus\bigcup_{j=1}^{[d^{-2}]}\Gamma_j\Big),$$
    \begin{align}
         \rchi_{\mathbf{\Gamma}} \, \mathbb{A} \cdot \frac{\partial^2}{\partial t^2} \mathbcal{F}_m (x, \mathrm{t}) + \mathbcal{F}_m (x, \mathrm{t}) 
         \nonumber&+ \sum\limits_{\substack{j=1\\j\ne m}}^{[d^{-2}]}\int_{\Gamma_j} \mathbcal{C}(x, y) \cdot \frac{\partial^2}{\partial t^2} \mathbcal{F}_m(y, t - c_0^{-1} \vert x - y \vert) \, dy = \frac{\partial^2}{\partial t^2} \mathbcal{F}_m^\textbf{in}(x, \mathrm{t})
         \\ \nonumber&- \int_{\Gamma_m} \mathbcal{C}(x, y) \cdot \frac{\partial^2}{\partial t^2} \mathbcal{F}_m(y, t - c_0^{-1} \vert x - y \vert) \, dy
         \\ &- \int_{\bm{\Gamma}\setminus\bigcup_{j=1}^{[d^{-2}]}\Gamma_j} \mathbcal{C}(x, y) \cdot \frac{\partial^2}{\partial t^2} \mathbcal{F}_m(y, t - c_0^{-1} \vert x - y \vert) \, dy
    \end{align}
Next, we use the Taylor's series expansion considering $\Big(z_{m_i}\Big)_{i=1}^{\lfloor K(\cdot) + 1 \rfloor}\in \Gamma_m$ and $\Big(z_{j_l}\Big)_{l=1}^{\lfloor K(\cdot) + 1 \rfloor}\in \Gamma_j$ for $m\ne j$. We also use that fact that
\begin{align}
    \textit{Area}(\Gamma_j) = d^2,
\end{align}
we arrive at the following expression in a discretized form at $x=z_m:$
    \begin{align}\label{dis}
         \nonumber&\rchi_{\mathbf{\Gamma}} \, \mathbb{A} \cdot \frac{\partial^2}{\partial t^2} \mathbcal{F}_m (z_m, \mathrm{t}) 
         + \mathbcal{F}_m (z_m, \mathrm{t}) 
         + \sum\limits_{\substack{j=1\\j\ne m}}^{[d^{-2}]}\mathbb{C}_{mj}(z_m, z_j) \cdot \frac{\partial^2}{\partial t^2} \mathbcal{F}_m(z_m, t - c_0^{-1} \vert z_m - z_j \vert) \, dy 
         = \frac{\partial^2}{\partial t^2} \mathbcal{F}_m^\textbf{in}(x, \mathrm{t})
         \\ \nonumber&- \underbrace{\int_{\Gamma_m} \mathbcal{C}(z_m, y) \cdot \frac{\partial^2}{\partial t^2} \mathbcal{F}_m(y, t - c_0^{-1} \vert z_m - y \vert) \, dy}_{:= \textit{Err}^{(1)}}
         - \underbrace{\int_{\bm{\Gamma}\setminus\bigcup\limits_{j=1}^{[d^{-2}]}\Gamma_j} \mathbcal{C}(z_m, y) \cdot \frac{\partial^2}{\partial t^2} \mathbcal{F}_m(y, t - c_0^{-1} \vert z_m - y \vert) \, dy}_{:= \textit{Err}^{(2)}}
        \\  &+ \underbrace{\sum\limits_{\substack{j=1\\j\ne m}}^{[d^{-2}]}\int_{\Gamma_j} \Bigg(\mathbb{C}_{mj}(z_m,z_j)\cdot\frac{\partial^2}{\partial t^2} \mathbcal{F}_m(z_m, t - c_0^{-1} \vert z_m - z_j \vert)-\mathbcal{C}(z_m, y)\cdot \frac{\partial^2}{\partial t^2} \mathbcal{F}_m(y, t - c_0^{-1} \vert z_m - y \vert) \Bigg)\, dy}_{:= \textit{Err}^{(3)}}.
    \end{align}
We start with estimating the error terms. 
\begin{enumerate}
    \item Estimation of $\textit{Err}^{(1)}.$ First, by local co-ordinate change, we observe that the      image of the ball with radius $r$, $B(x,r), \textit{Im}\big(B(x,r)\big)$ contained in $\Gamma_m.$ Then, we have
        \begin{align}
            \nonumber&|\textit{Err}^{(1)}| 
            \lesssim \big\Vert \frac{\partial^2}{\partial t^2} \mathbcal{F}_m \Vert_{L^\infty\big(0,\mathrm{T};\mathrm{L}^\infty(\Gamma_m))} \int_{\Gamma_m} \big|\mathbcal{C}(z_m, y)\big|d\sigma_y
            \\ \nonumber&\lesssim \big\Vert \frac{\partial^2}{\partial t^2} \mathbcal{F}_m \Vert_{L^\infty\big(0,\mathrm{T};\mathrm{L}^\infty(\Gamma_m))} \int_{\Gamma_m}|\Phi^{(0)}(z_m,y)|d\sigma_y
            \\ \nonumber&\lesssim \big\Vert \frac{\partial^2}{\partial t^2} \mathbcal{F}_m \Vert_{L^\infty\big(0,\mathrm{T};\mathrm{L}^\infty(\Gamma_m))} \Bigg(\int_{{\textit{Im}\big(B(z_m,r)\big)}}\frac{1}{|z_m-y|}dy + \int_{\Gamma_m\setminus{\textit{Im}\big(B(z_m,r)\big)}}\frac{1}{|z_m-y|}d\sigma_y\Bigg)
            \\ \nonumber&\lesssim \big\Vert \frac{\partial^2}{\partial t^2} \mathbcal{F}_m \Vert_{L^\infty\big(0,\mathrm{T};\mathrm{L}^\infty(\Gamma_m))} \Big(2\pi r + \frac{1}{r}(d^2- \pi r^2)\Big),\; \text{since}\; \textit{Area}\Big(\Gamma_m\setminus \textit{Im}\big(B(z_m,r)\big)\Big) = d^2- \pi r^2
            \\ \nonumber&\lesssim \big\Vert \frac{\partial^2}{\partial t^2} \mathbcal{F}_m \Vert_{L^\infty\big(0,\mathrm{T};\mathrm{L}^\infty(\Gamma_m))} \Big(2\pi r + \frac{1}{r}   (d^2- \pi r^2)\Big). 
        \end{align}
        Now, as this expression has a critical point at $r_{\text{sol}} = (\frac{1}{\pi})^\frac{1}{2}\;d,$ we further conclude after taking the square and summing from $1$ to $[d^{-2}]$ the following
        \begin{align}\label{ef-med-1}
              \sum\limits_{\substack{m=1}}^{[d^{-2}]} \big|\textit{Err}^{(1)}\big|^2 
              &\nonumber\lesssim \big\Vert \frac{\partial^2}{\partial t^2} \mathbcal{F}_m \Vert_{L^\infty\big(0,\mathrm{T};\mathrm{L}^\infty(\Gamma_m))}^2 d^2 d^{-2}
              \\ &\lesssim d^2\sum\limits_{\substack{m=1}}^{[d^{-2}]}\big\Vert \frac{\partial^2}{\partial t^2} \mathbcal{F}_m \Vert_{L^\infty\big(0,\mathrm{T};\mathrm{L}^\infty(\mathbf{\Gamma}))}^2.
        \end{align}
    \item Estimation of $\textit{Err}^{(2)}.$ In order to estimate this term, we first split the           region of the integral, $\bm{\Gamma}\setminus\bigcup_{j=1}^{[d^{-2}]}\Gamma_j$ into two            sub-region for any fixed $j.$ Let us do that in the following way.
        \begin{enumerate}
            \item Let us denote $\aleph_{(1)}\subset\bm{\Gamma}\setminus\bigcup_{j=1}                        ^{[d^{-2}]}\Gamma_j$ in such a way that for any $z_j\in \aleph_{(1)},$ which locates       away from $y\in S_j,$ where $S_j$ is the largest ball containing $z_j.$ In that            case, the function $|z_i - y|^{-1}$ remains bounded in the vicinity of  $S_j$.             Consequently, we derive that $\textit{Area}\Big(\aleph_{(1)}\Big) =                         \textit{Area}\Big(\bm{\Gamma}\setminus\bigcup_{j=1}^{[d^{-2}]}\Gamma_j\Big) =               \mathcal{O}(d).$
            \item Let us now denote $\aleph_{(2)}\subset\bm{\Gamma}\setminus\bigcup_{j=1}                    ^{[d^{-2}]}\Gamma_j$ be such that $y\in S_j$ locates near one of the $\Gamma_j$'s          which touches the boundary $\partial\bm{\Gamma}$ of $\aleph_{(2)}.$
        \end{enumerate}

\begin{figure}
\begin{center}
\begin{tikzpicture}[scale=1.2]

\draw (0,0) ellipse (3cm and 1.5cm);
\node[scale=1] at (0,-1.8) {$\mathbf{\Gamma}$};

\draw[<-,red, thick] (1.6, 1.15) -- (2.5, 1.5);
\node[above, scale=0.5] at (2.5,1.5) {$\aleph_{(1)}\ (\text{The region above the blue line})$};

\draw[<-,red, thick] (-1.6, -1.15) -- (-2.5, -1.5);
\node[above, scale=0.5] at (-2.5,-1.75) {$\aleph_{(2)} (\text{The region below the blue line})$};

\draw[<-,red, thick] (0.1,1.2) -- (1, 1.8);
\node[above, scale=0.5] at (1,1.8) {$\mathrm{D}_i$};

\node[above, scale=.5] at (0,0.8) {$z_i$};
\node[above, scale=0.5] at (0,0.9) {$\bullet$};

\coordinate (B) at (0,1); 
\draw (B) circle (0.2);

\draw[dotted, blue, line width=1.2pt] (-3.5,0.5) -- (3.5,0.5);

\clip (0,0) ellipse (3cm and 1.5cm);

\foreach \x in {-3.5,-2.5,...,3.5} {
    \draw (\x,-2) -- (\x,2);
}
\def\xellip(#1){4*sqrt(1 - (#1/2)^2)}

\foreach \y in {-1.75,-1.25,...,1.75} {
    \draw ({\xellip(\y)},\y) -- ({-\xellip(\y)},\y);
}

\begin{scope}
    \clip (2.5,-1.5) rectangle (3.5,1.5);
    \foreach \i in {-1.5,-1.3,...,3.5} {
        \draw[black, thick] (\i,-0.8) -- ++(3.5,1.5);
    }
\end{scope}
\begin{scope}
    \clip (-2.9,-1.5) rectangle (3.5,-1.25);
    \foreach \i in {-2.9,-2.65,...,0} {
        \draw[black, thick] (\i,-1.7) -- ++(3.5,1);
    }
\end{scope}

\begin{scope}
    \clip (-3.9,-1.5) rectangle (-2.5,2.7);
    \foreach \i in {-3.9,-3.8,...,3.9} {
        \draw[black, thick] (\i,-.9) -- ++(3.5,4.9);
    }
\end{scope}

\begin{scope}
    \clip (-3,1) rectangle (5.5,2);
    \foreach \i in {-3,-2.8,...,2} {
        \draw[black, thick] (\i,1.25) -- ++(5.5,1.5);
    }
\end{scope}

\begin{scope}
    \clip (1.5,0.75) rectangle (5.5,1.5);
    \foreach \i in {-3,-2.8,...,2} {
        \draw[black, thick] (\i,0) -- ++(5.5,1.5);
    }
\end{scope}

\begin{scope}
    \clip (-6.5,0.75) rectangle (-1.5,1.5);
    \foreach \i in {-6.5,-6.3,...,2} {
        \draw[black, thick] (\i,0) -- ++(5.5,1.5);
    }
\end{scope}

\begin{scope}
    \clip (-6.5,-2) rectangle (-1.5,-.75);
    \foreach \i in {-6.5,-6.3,...,2} {
        \draw[black, thick] (\i,-.5) -- ++(5.5,-2);
    }
\end{scope}

\begin{scope}
    \clip (1.5,-.75) rectangle (5.5,-5);
    \foreach \i in {-3,-2.8,...,5} {
        \draw[black, thick] (\i,0) -- ++(5.5,-3);
    }
\end{scope}

\draw (0,0) ellipse (3cm and 1.5cm);

\end{tikzpicture}
\end{center}
\caption{A schematic illustration for the split of the region $\mathbf{\Gamma}\setminus\bigcup\limits_{j=1}^{[d^{-2}]}\Gamma_j.$ } \label{pic1}
\end{figure}
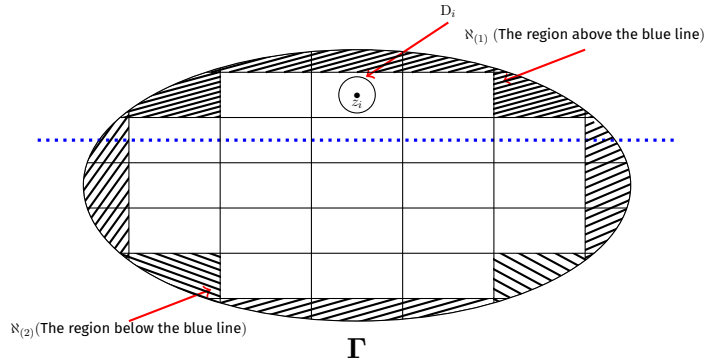    

    Therefore, we have
        \begin{align}
            |\textit{Err}^{(2)}| 
            &\nonumber=\Big|\int_{\bm{\Gamma}\setminus\bigcup_{j=1}^{[d^{-2}]}\Gamma_j} \mathbcal{C}(x, y) \cdot \frac{\partial^2}{\partial t^2} \mathbcal{F}_m(y, t - c_0^{-1} \vert x - y \vert) \, dy\Big|
            \\ &\le \Big|\int_{\aleph_{(1)}} \mathbcal{C}(x, y) \cdot \frac{\partial^2}{\partial t^2} \mathbcal{F}_m(y, t - c_0^{-1} \vert x - y \vert) \, dy\Big| + \Big|\int_{\aleph_{(2)}} \mathbcal{C}(x, y) \cdot \frac{\partial^2}{\partial t^2} \mathbcal{F}_m(y, t - c_0^{-1} \vert x - y \vert) \, dy\Big|
        \end{align} 
        Then, by utilizing Lemma \ref{counting}, that provides the counting criterion near $\partial\bm{\Gamma},$ we derive
        \begin{align}\label{ll1}
            |\textit{Err}^{(2)}| 
            \nonumber&\lesssim \big\Vert \frac{\partial^2}{\partial t^2} \mathbcal{F}_m \Vert_{L^\infty\big(0,\mathrm{T};\mathrm{L}^\infty(\Gamma_j))}\ \textit{Area}(\Gamma_j)\sum\limits_{\substack{j=1\\j\ne m}}^{[d^{-1}]} |\Phi^{(0)}(z_j,y)| +  \big\Vert \frac{\partial^2}{\partial t^2} \mathbcal{F}_m \Vert_{L^\infty\big(0,\mathrm{T};\mathrm{L}^\infty(\Gamma_j))}\ \textit{Area}(\aleph^{(1)})
            \\ \nonumber &\lesssim \Big(d^2\sum\limits_{\substack{j=1}}^{[d^{-1}]}\frac{1}{d_{ij}} + d\Big)\big\Vert \frac{\partial^2}{\partial t^2} \mathbcal{F}_m \Vert_{L^\infty\big(0,\mathrm{T};\mathrm{L}^\infty(\Gamma_j))}
            \\ \nonumber &\lesssim \Big(d^2 \sum\limits_{\substack{n=1}}^{[d^{-1}]}\big[(2n+1)-(2n-1)\big]\frac{1}{n(d-\frac{d^2}{2})} 
            + d\Big)\big\Vert \frac{\partial^2}{\partial t^2} \mathbcal{F}_m \Vert_{L^\infty\big(0,\mathrm{T};\mathrm{L}^\infty(\Gamma_j))}
            \\  &\lesssim \Big(d^2\cdot d^{-1}|\log(d)| + d\Big)\big\Vert \frac{\partial^2}{\partial t^2} \mathbcal{F}_m \Vert_{L^\infty\big(0,\mathrm{T};\mathrm{L}^\infty(\Gamma_j))} \lesssim d\ \big\Vert \frac{\partial^2}{\partial t^2} \mathbcal{F}_m \Vert_{L^\infty\big(0,\mathrm{T};\mathrm{L}^\infty(\Gamma_j)\big)}.
        \end{align}
        Finally, we conclude after taking the square and summing from $1$ to $[d^{-2}]$ the following
        \begin{align}\label{ef-med-2}
              \sum\limits_{\substack{m=1}}^{[d^{-2}]} \big|\textit{Err}^{(2)}\big|^2 \lesssim
              d^2\sum\limits_{\substack{m=1}}^{[d^{-2}]}\big\Vert \frac{\partial^2}{\partial t^2} \mathbcal{F}_m \Vert_{L^\infty\big(0,\mathrm{T};\mathrm{L}^\infty(\Gamma_m))}^2
        \end{align}      
    \item Estimation of $\textit{Err}^{(3)}.$ This term can be rewritten as follows
    \begin{align}
        \textit{Err}^{(3)} 
        \nonumber&:= \underbrace{\sum\limits_{\substack{j=1\\j\ne m}}^{[d^{-2}]}\int_{\Gamma_j} \Big(\mathbb{C}_{mj}(z_m,z_j)-\mathbcal{C}(z_m, y)\Big) \cdot \frac{\partial^2}{\partial t^2} \mathbcal{F}_m(y, t - c_0^{-1} \vert z_m - y \vert) \, dy}_{:= \textit{Err}^{(31)}}
        \\ &+ \underbrace{\sum\limits_{\substack{j=1\\j\ne m}}^{[d^{-2}]}\int_{\Gamma_j} \mathbb{C}_{mj}(z_m,z_j)\cdot\Big(\frac{\partial^2}{\partial t^2} \mathbcal{F}_m(z_m, t - c_0^{-1} \vert z_m - z_j \vert)- \frac{\partial^2}{\partial t^2} \mathbcal{F}_m(y, t - c_0^{-1} \vert z_m - y \vert)\Big)dy}_{:= \textit{Err}^{(32)}}
    \end{align}
        \begin{align}
    \nonumber&\textit{Err}^{(31)} 
    \\ \nonumber&:= \sum\limits_{\substack{j=1\\j\ne m}}^{[d^{-2}]}\int_{\Gamma_j} \Big(\mathbb{C}_{mj}(z_m,z_j)-\mathbcal{C}(z_m, y)\Big) \cdot \frac{\partial^2}{\partial t^2} \mathbcal{F}_m(y, t - c_0^{-1} \vert z_m - y \vert) \, dy
    \\ &\lesssim \sum\limits_{\substack{j=1\\j\ne m}}^{[d^{-2}]}\int_{\Gamma_j} 
    \begin{pmatrix}
        \rchi_{m_1,j_1} & \rchi_{m_1,j_2} & \ldots & \rchi_{m_1,j_{\lfloor K(\cdot) + 1 \rfloor}}\\
        \vdots & \vdots & \ddots & \vdots \\
        \rchi_{m_{\lfloor K(\cdot) + 1} \rfloor,j_{\lfloor K(\cdot) + 1 \rfloor}} & \rchi_{m_{\lfloor K(\cdot) + 1},j_{\lfloor K(\cdot) + 1}} & \ldots & \rchi_{m_{\lfloor K(\cdot) + 1},j_{\lfloor K(\cdot) + 1 \rfloor}}
    \end{pmatrix}\cdot \frac{\partial^2}{\partial t^2} \mathbcal{F}_m(y, t - c_0^{-1} \vert z_m - y \vert) \, dy,
\end{align}
where, we define
\begin{align*}
    \rchi_{m_i,j_l} = \Phi^{(0)}(z_{m_i},y) - \Phi^{(0)}(z_{m_i},z_{j_l}),\; \textit{for}\; 1\le i,l\le \lfloor K(\cdot) + 1 \rfloor.
\end{align*}
Consequently, we arrive at the following expression
\begin{align}\label{er31}
    \nonumber&\big| \textit{Err}^{(31)} \big| 
    \\ \nonumber&\lesssim \big\Vert \frac{\partial^2}{\partial t^2} \mathbcal{F}_m \Vert_{L^\infty\big(0,\mathrm{T};\mathrm{L}^\infty(\mathbf{\Gamma}))}\sum\limits_{\substack{j=1\\j\ne m}}^{[d^{-2}]}\int_{\Gamma_j}\Bigg(\sum\limits_{i=1}^{\lfloor K(\cdot) + 1\rfloor}\sum\limits_{l=1}^{\lfloor K(\cdot) + 1\rfloor}\big|\Phi^{(0)}(z_{m_i},y) - \Phi^{(0)}(z_{m_i},z_{j_l})\big|^2\Bigg)^\frac{1}{2}dy
    \\ \nonumber&\lesssim \big\Vert \frac{\partial^2}{\partial t^2} \mathbcal{F}_m \Vert_{L^\infty\big(0,\mathrm{T};\mathrm{L}^\infty(\mathbf{\Gamma}))}\sum\limits_{\substack{j=1\\j\ne m}}^{[d^{-2}]}\int_{\Gamma_j}\Bigg(\sum\limits_{i=1}^{\lfloor K(\cdot) + 1\rfloor}\sum\limits_{l=1}^{\lfloor K(\cdot) + 1\rfloor}\Big|\int_0^1\nabla\Phi^{(0)}(z_{m_i}, z_{j_l}+\theta(y-z_{j_l}))\cdot (y-z_{j_l})d\theta\Big|^2\Bigg)^\frac{1}{2}dy
    \\ &\lesssim \big\Vert \frac{\partial^2}{\partial t^2} \mathbcal{F}_m \Vert_{L^\infty\big(0,\mathrm{T};\mathrm{L}^\infty(\mathbf{\Gamma}))}\ d^3 \sum\limits_{\substack{j=1\\j\ne m}}^{[d^{-2}]} d_{mj}^{-1} 
    \lesssim d\ \big\Vert \frac{\partial^2}{\partial t^2} \mathbcal{F}_m \Vert_{L^\infty\big(0,\mathrm{T};\mathrm{L}^\infty(\mathbf{\Gamma}))}.
\end{align}
Next, following a similar approach as used to estimate the terms \(\textit{Error}^{(6)}(\ref{error6})\) and \(\textit{Error}^{(7)}(\ref{error7})\), we can deduce that  
\begin{align}\label{er32}
    \big| \textit{Err}^{(32)} \big| \lesssim d\ \big\Vert \frac{\partial^2}{\partial t^2} \mathbcal{F}_m \Vert_{L^\infty\big(0,\mathrm{T};\mathrm{L}^\infty(\mathbf{\Gamma}))}.
\end{align}
Subsequently, using Counting Lemma \ref{counting}, we deduce the following estimate by squaring the estimates (\ref{er31}) and (\ref{er32}) and summing over the range from \( 1 \) to \( [d^{-2}] \)
\begin{align}\label{ef-med-3}
    \sum\limits_{m=1}^{[d^{-2}]} \big|\textit{Err}^{(3)}\big|^2 \lesssim d^2\sum\limits_{\substack{m=1}}^{[d^{-2}]}\big\Vert \frac{\partial^2}{\partial t^2} \mathbcal{F}_m \Vert_{L^\infty\big(0,\mathrm{T};\mathrm{L}^\infty(\Gamma_m))}^2.
\end{align}
\end{enumerate}
Subsequently, based on the estimates (\ref{ef-med-1}), (\ref{ef-med-2}) and (\ref{ef-med-3}), we derive that
    \begin{align}\label{f-es}
        \sum\limits_{\substack{m=1}}^{[d^{-2}]} \Big(\big|\textit{Err}^{(3)}\big|^2 + \big|\textit{Err}^{(3)}\big|^2 + \big|\textit{Err}^{(3)}\big|^2\Big)\lesssim 
        d^2\sum\limits_{\substack{m=1}}^{[d^{-2}]}\big\Vert \frac{\partial^2}{\partial t^2} \mathbcal{F}_m \Vert_{L^\infty\big(0,\mathrm{T};\mathrm{L}^\infty(\Gamma_m))}^2.
    \end{align}
Consequently, by comparing the discretized expression (\ref{dis}) and the general algebric system (\ref{tran1}), we arrive at the following system with $\mathbcal{Z}_m(z_m,t) := \bm{\mathrm{Y}}_m(t)-\mathbcal{F}_m(z_m,t)$
    \begin{align} \label{matrix2}
        \begin{cases}
             \mathcal{A}\frac{\mathrm{d}^2}{\mathrm{d}\mathrm{t}^2}\bm{\mathbcal{Z}}_m(z_m,\mathrm{t}) + \bm{\mathbcal{Z}}_m(z_m,\mathrm{t}) =  \mathcal{O}\Big(d
             \ \big\Vert \frac{\partial^2}{\partial t^2} \mathbcal{F}_m \Vert_{L^\infty\big(0,\mathrm{T};\mathrm{L}^\infty(\Gamma_m))}\Big) \mbox{ in } (0, \mathrm{T}),
             \\ \bm{\mathbcal{Z}}_m(x,\mathrm{0}) = \frac{\mathrm{d}}{\mathrm{d}\mathrm{t}}\bm{\mathbcal{Z}}_m(x,\mathrm{0}) = 0,   
        \end{cases}
    \end{align}
where, we define the operator $\mathcal{A}: (\mathrm{L}_r^2)^M\to (\mathrm{L}_r^2)^M$ as
\begin{align}
    \mathcal{A} = \mathcal{A}(t):=
     \begin{pmatrix}
        \omega_M^2 & \dots  & \frac{\mathbf{C}_{j_{\lfloor K(\cdot) + 1 \rfloor}}}{4\pi |z_{m_1} - z_{j_{\lfloor K(\cdot) + 1 \rfloor}|}}\mathcal{T}_{-\mathrm{c}_0^{-1}|z_{m_1} - z_{j_{\lfloor K(\cdot) + 1 \rfloor}|}}\\
       \vdots & \ddots & \vdots\\
        \frac{\mathbf{C}_{j_1}}{4\pi |z_{m_{\lfloor K(\cdot) + 1 \rfloor}}-z_{j_1}|} \mathcal{T}_{-\mathrm{c}_0^{-1}|z_{m_{\lfloor K(\cdot) + 1 \rfloor}}-z_{j_1}|}& \dots  & \omega_M^2
    \end{pmatrix}, 
\end{align}
with $\mathrm{L}_\ell^2:= \{ f \in \mathrm{L}^2(-r,\mathrm{T}): f=0\; \text{in}\ (-r,0)\}$, the translation operators $\mathcal{T}_{-\mathrm{c}_0^{-1}|\mathrm{z}_i-\mathrm{z}_j|}$, i.e. $\mathcal{T}_{-\mathrm{c}_0^{-1}|\mathrm{z}_i-\mathrm{z}_j|}(f)(t):=f(t-\mathrm{c}_0^{-1}|\mathrm{z}_i-\mathrm{z}_j|)$, and $\ell:=\max_{i\neq j}{\mathrm{c}^{-1}_0\vert z_i- z_j\vert}$.
\newline
Hence, we use the well-possedness of the problem (\ref{matrix2}) as discussed in \cite[Section 2.4]{Arpan-Sini-JEEQ} and the estimate (\ref{f-es}) to obtain
\begin{align}\label{mainesti}
    \sum_{m=1}^{[d^{-2}]}|\mathbcal{F}_m(z_m,t)-\bm{\mathbb{Y}}_m(t)|^2 = \mathcal{O}\Big(d^2\sum_{m=1}^{[d^{-2}]}\big\Vert \frac{\partial^2}{\partial t^2} \mathbcal{F}_m \Vert_{L^\infty\big(0,\mathrm{T};\mathrm{L}^\infty(\Gamma_m)\big)}^2\Big),\; \text{as}\; d\ll 1.
\end{align}
It completes the proof. \qed


   \section{Appendix: Counting Lemmas} 

\noindent
In this section, we focus on key results that are instrumental in the proofs of our main theorems, specifically for calculating the sums of inverse distances between inclusions located within the region \(\bm{\Gamma}\) and near its boundary, \(\partial\bm{\Gamma}\).
\begin{lemma}\cite{habib-sini} \label{counting}
    \textit{Counting Lemma.} For an arbitrary distribution of points \( z_j \) (\( j=1, \dots, M \)) within a bounded domain in \(\mathbb{R}^2\), where the minimum distance between points is given by \( d \), the following estimates hold uniformly with respect to each point \( z_i \):
    \begin{align}
        \sum_{\substack{j=1 \\ j\neq i}}^{M} \frac{1}{|z_i - z_j|^k} = 
        \begin{cases}
            \mathcal{O}(d^{-2}), & \text{if}\; k < 2, \\
            \mathcal{O}\left(d^{-2}(1 + |\log(d)|)\right), & \text{if}\; k = 2, \\
            \mathcal{O}(d^{-k}), & \text{if}\; k > 2.
        \end{cases}
    \end{align}
\end{lemma}
\noindent
The implications of these asymptotic estimates are crucial for controlling the interactions between inclusions as they relate to the minimum distance \( d \) in the domain. Based on the above lemma, we make the following remarks:
\begin{remark}
    We consider \(\bm{\Gamma}\) to have a sufficiently smooth boundary, \(\partial\bm{\Gamma}\), and assume \(\bm{\Gamma}\) has unit surface area. We aim to estimate the total area of squares \(\Gamma_j\) that intersect with \(\partial\bm{\Gamma}\). First, note that the maximum area of each square \(\Gamma_j\) is on the order of \(d^2\), and the radius of each \(\Gamma_j\) is on the order of \(d\). Consequently, the length of each curve segment where \(\Gamma_j\) intersects \(\partial\bm{\Gamma}\) is on the order of \(d\). Given that the total length of \(\partial\bm{\Gamma}\) is of order 1, the number of such intersecting squares \(\Gamma_j\) cannot exceed an order of \(d^{-1}\). Therefore, the total area of the intersecting squares is bounded above by \(d^2 \cdot d^{-1} = d\), ensuring that the area of this intersecting set remains within an order of \(d\).
\end{remark}
\begin{remark}
    In this remark, we outline a method to estimate the number of inclusions near the boundary \(\partial\bm{\Gamma}\). Let \( z_j \) be a fixed arbitrary point representing the location of one of the inclusions, and assume that \( z_j \) is close to a square \(\Gamma_j\) that intersects the boundary \(\partial\bm{\Gamma}\). To count the number of inclusions in the region \(\bm{\Gamma} \setminus \cup_{j=1}^M \Gamma_j\), we first approximate any small segment of \(\partial\bm{\Gamma}\) as flat. We then partition the region in \(\bm{\Gamma} \setminus \cup_{j=1}^M \Gamma_j\), near this approximate flat segment, into concentric layers formed by squares. Given that the total length of \(\partial\bm{\Gamma}\) is of order 1, the number of intersecting squares \(\Gamma_j\) along this boundary cannot exceed an order of \(d^{-1}\). Consequently, we have a maximum of \((2n + 1)\) small squares for each \( n = 0, 1, \ldots, [d^{-1}]\) intersecting with \(\partial\bm{\Gamma}\). For each \(n^{\textit{th}}\) layer, the number of inclusions within the layer is bounded by the order of \(\big[(2n + 1) - (2n - 1)\big]\), while their minimum distance from \(\Gamma_j\) is approximately \(n(d - \frac{d^2}{2})\).
\end{remark}

\end{document}